\documentclass{article}
\usepackage{amsmath, amsfonts, amsthm, amssymb}
\usepackage{paralist}
\usepackage{anyfontsize}
\usepackage{graphics} 
\usepackage{epsfig} 
\usepackage{graphicx}  
\usepackage{multicol}
\usepackage{xcolor}

\usepackage{url}
\usepackage{pdfpages}

\usepackage{tikz} 
\usetikzlibrary{shapes,arrows, patterns,shapes.geometric}

\newtheorem{theorem}{Theorem}[section]
\newtheorem{proposition}[theorem]{Proposition}
\newtheorem{lemma}[theorem]{Lemma}
\newtheorem{corollary}[theorem]{Corollary}
\newtheorem{observation}[theorem]{Observation}
\newtheorem{example}[theorem]{Example}
\newtheorem{ques}[theorem]{Question}

\theoremstyle{definition}
\newtheorem{definition}[theorem]{Definition}

\usepackage[colorlinks=true, allcolors=blue]{hyperref}
\definecolor{royalazure}{rgb}{0.0, 0.22, 0.66}

\title{Characterizing Graphs as Algebraic Squares}

\author{
Karen L.\ Collins\thanks{Department of Mathematics and Computer Science, Wesleyan University, Middletown CT 06459, 
{\tt kcollins@wesleyan.edu}}
\and
David Galvin\thanks{Department of Mathematics, University of Notre Dame, Notre Dame IN 46556, {\tt dgalvin1@nd.edu}. Galvin is in part supported by a Simons Collaboration Grant for Mathematicians.}
\and
Christine A.\ Kelley\thanks{Department of Mathematics, University of Nebraska-Lincoln, Lincoln NE 68588, 
{\tt ckelley2@nebraska.edu}}
\and
Emily McMillon\thanks{Department of Mathematics, Rice University, Houston TX 77005, {\tt em72@rice.edu}. McMillon is in part supported by NSF DMS-2303380.}
\and
Amanda Redlich\thanks{Department of Mathematics and Statistics, University of Massachussetts-Lowell, Lowell MA 01834, {\tt amanda\_redlich@uml.edu}}
}

\begin{document}

\maketitle

\begin{abstract}

Graphs that are squares under the gluing algebra arise in the study of homomorphism density inequalities such as Sidorenko's conjecture.  Recent work has focused on these homomorphism density applications.  This paper takes a new perspective and focuses on the graph properties of arbitrary square graphs, not only those relevant to homomorphism conjectures and theorems. 

We develop a set of necessary and/or sufficient conditions for a graph to be square.  We apply these conditions to categorize several classical families of graphs as square or not. In addition, we  create infinite families of square graphs by proving that joins and Cartesian, direct, strong, and lexicographic products of  square graphs with arbitrary graphs are square.
\end{abstract}

\section{Introduction} \label{sec:intro}

In this paper we consider graphs which are ``squares", i.e. the product of a graph with itself, under the gluing algebra (see \cite{limits} for an extensive discussion of this algebra).  This particular algebra has been studied most recently in the context of homomorphism densities (\cite{tropical}, \cite{path}, \cite{BRST20}, \cite {AAP}).  In a series of papers, connections between homomorphism density inequalities (e.g. Sidorenko's conjecture) and the gluing algebra were explored. In particular, the notion of graph as squares was used to show that certain expressions in the gluing algebra are positive. This positivity implies important homomorphism density inequalities. 

Squares have also appeared in an integral way in the context of the Positive Graphs conjecture (see e.g. \cite{ANTOLINCAMARENA2016290}, \cite{CLV2024}). Positivity is a condition related to the sign of weighted homomorphism counts, and the conjecture posits a characterization of graphs being positive in terms of them being a particular type of square in the gluing algebra.

Since much of the previous work on square graphs has focused on applications to homomorphism density inequalities, specific families of graphs that are naturally relevant to Sidorenko's inequality and others were studied. However, many classical graph families (e.g. any non-bipartite graphs) do not fall into this category.  We present results about such families below. We focus on the inherent properties of square graphs, compared to other graphs.

This classification is of interest in its own right.  The definition of a square graph from an algebraic perspective is difficult to implement for natural graph structures.  However, it is in fact a very natural property.  When a graph is square, it contains a ``core" of fixed vertices and two identical ``left and right" subgraphs outside the core; one visualization is to think of it as a butterfly, with the fixed vertices forming the thorax and the left and right subgraphs forming the left and right wing. Figure~\ref{fig:butterfly} illustrates a graph $B$ with two vertices (red) in its core, a black left wing and a gray right wing. This type of symmetry is of course common in nature and aesthetically pleasing; it is only natural to search for it in graphs as well.

\begin{figure}[htb]
    \centering
    \begin{tikzpicture}[scale=0.8, every node/.style={scale=0.8}]
        \draw[black,thick,dashed] (0,1.5) to (0,-4.75);
        
        \node[draw=black,circle,fill=red] (1) at (0,0) {};
        \node[draw=black,circle,fill=gray] (2) at (0.5,0.5) {};
        \node[draw=black,circle,fill=gray] (3) at (1,1) {};
        \node[draw=black,circle,fill=black] (4) at (-0.5,0.5) {};
        \node[draw=black,circle,fill=black] (5) at (-1,1) {};
        \node[draw=black,circle,fill=gray] (6) at (0.5,-1) {};
        \node[draw=black,circle,fill=gray] (7) at (0.5,-2) {};
        \node[draw=black,circle,fill=red] (8) at (0,-3) {};
        \node[draw=black,circle,fill=black] (9) at (-0.5,-1) {};
        \node[draw=black,circle,fill=black] (10) at (-0.5,-2) {};
        \node[draw=black,circle,fill=gray] (11) at (1,0) {};
        \node[draw=black,circle,fill=gray] (12) at (2,-1) {};
        \node[draw=black,circle,fill=gray] (13) at (2,-2) {};
        \node[draw=black,circle,fill=gray] (14) at (1,-3) {};
        \node[draw=black,circle,fill=gray] (15) at (1.25,-3.75) {};
        \node[draw=black,circle,fill=gray] (16) at (0.25,-4.25) {};

        \node[draw=black,circle,fill=black] (17) at (-1,0) {};
        \node[draw=black,circle,fill=black] (18) at (-2,-1) {};
        \node[draw=black,circle,fill=black] (19) at (-2,-2) {};
        \node[draw=black,circle,fill=black] (20) at (-1,-3) {};
        \node[draw=black,circle,fill=black] (21) at (-1.25,-3.75) {};
        \node[draw=black,circle,fill=black] (22) at (-0.25,-4.25) {};

        \draw[black,thick] (3) -- (2) -- (1) -- (4) -- (5);
        \draw[black,thick] (1) -- (6) -- (7) -- (8) -- (10) -- (9) -- (1);
        \draw[black,thick] (1) -- (11) -- (12) -- (13) -- (14) -- (8) -- (15) -- (16) -- (8);
        \draw[black,thick] (1) -- (17) -- (18) -- (19) -- (20) -- (8) -- (21) -- (22) -- (8);
    \end{tikzpicture}
    \caption{A graph $B$ with a butterfly involution through the vertical line containing the two vertices of degree 6, colored red. Thus, $B$ is a square graph whose labeled vertices are the red vertices. Each gray vertex is paired with a black vertex. The line of symmetry of the graph is marked with a dashed line.} 
    \label{fig:butterfly}
\end{figure}
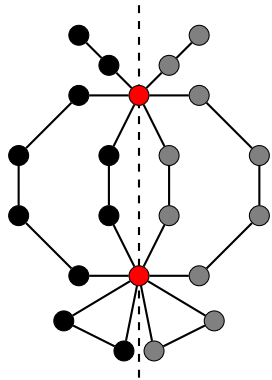

Our first contributions, in Sections \ref{sec:prelims} and \ref{sec:general}, are to develop formally equivalent definitions of a square graph that are more amenable to proof via structural methods.  We show a graph is square if and only if it contains a certain type of induced subgraph, and if and only if it allows for a ``butterfly" involution. We also observe a few necessary or sufficient conditions for a square related to the neighborhoods of individual vertices.

We show the practicality of these observations and definitions in Section \ref{sec:families}. There we quickly and easily classify as square or non-square several natural classes of graphs: cycles, paths, complete graphs, and wheels, and identify their square roots.  (Note that it is possible for a graph to have multiple non-isomorphic square roots, as observed in \cite{AAP}.)  We then use more complex methods to classify circulant and Johnson graphs as square or not and identify their square roots.

In Section \ref{sec:build-graphs}, we turn to methods for generating new square graphs from previously-classified graphs.  We prove that hypercubes are square, then generalize the argument to show any product of a square with an arbitrary graph is square. This is true for the Cartesian product, the tensor (direct) product, the strong product, the lexicographic product, and the join.  However, it is not true that this is a necessary and sufficient condition; we give examples of two non-square graphs whose product (under various of these definitions) are nevertheless square. We draw conclusions and discuss open problems in Section~\ref{sec:conclusion}.

\section{Preliminaries} \label{sec:prelims}

A graph $G$ has vertex set $V(G)$ and edge set $E(G)$, denoted $G = (V,E)$. Unless otherwise noted, all graphs are simple, without loops, and without multiple edges.

We begin by introducing the \textit{gluing algebra}. A \textit{partial labeling} of $G$ for $L \subseteq \mathbb{N}$ is an injective and non-surjective map $\theta: L \to V(G)$. (A \textit{labeling} is a surjective such map.)  We say that $G$ is \textit{partially labeled} if $\theta(L)$ is not empty, and that $G$ is \textit{unlabeled} if $\theta(L)$ is the empty set. (Similarly, $G$ is \textit{labeled} if $\theta(L)=V(G)$.)

Multiplication in this algebra is defined as follows. Let $H_1$ and $H_2$ be two (partially) labeled graphs. The new (partially) labeled graph $H_1 H_2$ is formed by identifying vertices with the same label and retaining their edges, except that any resulting double edge is replaced by a single edge. We give two examples of this operation in Figure~\ref{fig:gluing}. The gluing operation in (a) has a labeled vertex in the graph on the right which is not identified with any vertex in the graph on the left, while the resulting graph in (b) has a butterfly involution. Although the vertex labels are retained after multiplication, in this paper, the labels in the resulting graph are not of interest and are typically ignored. For ease of notation we choose to write the unlabeled product of two partially labeled graphs $G_1$ and $G_2$ as $G_1G_2$. 

\begin{figure}[htb]
    \centering
    \begin{tikzpicture}[scale=.8]
    \node at (.2,-2.3) {$G_1$};
     \node at (3.1,-2.3) {$G_2$};
     \node at (6.3,-2.3) {$G_1G_2$};
    \node at (-1.7,-.5) {(a)};
        \node[draw=black,circle,fill=blue,scale=0.8,label=above:{$1$}] (1) at (0,0) {};
        \node[draw=black,circle,fill=red,scale=0.8,label=below:{$2$}] (2) at (1,-1) {};
        \node[draw=black,circle,fill=black,scale=0.8] (a) at (-1,-1) {};
        \draw[black,thick] (1) -- (2) -- (a) -- (1);
        
        \node[draw=black,circle,fill=black,scale=0.2] at (2,-.5) {};

        \node[draw=black,circle,fill=black,scale=0.8] (b) at (3,1.8) {};
        \node[draw=black,circle,fill=blue,scale=0.8,label=right:{$1$}] (1b) at (3,0.8) {};
        \node[draw=black,circle,fill=red,scale=0.8,label=right:{$2$}] (2b) at (3,-0.2) {};
        \node[draw=black,circle,fill=yellow,scale=0.8,label=right:{$3$}] (3) at (3,-1.2) {};
        \draw[black,thick] (b) -- (1b) -- (2b) -- (3);

        \node at (4.3,0) {\Huge $=$};

        \node[draw=black,circle,fill=blue,scale=0.8,label=above:{$1$}] (1p) at (6,0) {};
        \node[draw=black,circle,fill=red,scale=0.8,label=below:{$2$}] (2p) at (7,-1) {};
        \node[draw=black,circle,fill=black,scale=0.8] (app) at (5,-1) {};
        \node[draw=black,circle,fill=yellow,scale=0.8,label=right:{$3$}] (3p) at (8,0) {};
        \node[draw=black,circle,fill=black,scale=0.8] (bp) at (7,1) {};
        \draw[black,thick] (1p) -- (2p) -- (app) -- (1p);
        \draw[black,thick] (1p) -- (bp);
        \draw[black,thick] (2p) -- (3p);
    \end{tikzpicture}

    \begin{tikzpicture}[scale=.8]
     \node at (.2,-2.4) {$H_1$};
     \node at (3.1,-2.4) {$H_2$};
     \node at (7,-2.4) {$H_1H_2$};

    \node at (-2,-.5) {(b)};
        \node[draw=black,circle,fill=blue,scale=0.8,label=above:{$1$}] (1) at (0,0) {};
        \node[draw=black,circle,fill=red,scale=0.8,label=below:{$2$}] (2) at (1,-1) {};
        \node[draw=black,circle,fill=black,scale=0.8] (a) at (-1,-1) {};
        \draw[black,thick] (1) -- (2) -- (a) -- (1);
        
        \node[draw=black,circle,fill=black,scale=0.2] at (1.5,-0.5) {};

        \node[draw=black,circle,fill=blue,scale=0.8,label=above:{$1$}] (1p) at (3,0) {};
        \node[draw=black,circle,fill=red,scale=0.8,label=below:{$2$}] (2p) at (4,-1) {};
        \node[draw=black,circle,fill=black,scale=0.8] (ap) at (2,-1) {};
        \draw[black,thick] (1p) -- (2p) -- (ap) -- (1p);

        \node at (5,-0.5) {\Huge $=$};

        \node[draw=black,circle,fill=black,scale=0.8] (b) at (6,-0.5) {};
        \node[draw=black,circle,fill=blue,scale=0.8,label=above:{$1$}] (1pp) at (7,0.5) {};
        \node[draw=black,circle,fill=red,scale=0.8,label=below:{$2$}] (2pp) at (7,-1.5) {};
        \node[draw=black,circle,fill=black,scale=0.8] (bpp) at (8,-0.5) {};
        \draw[thick,black] (1pp) -- (2pp) -- (bpp) -- (1pp) -- (b) -- (2pp);
    \end{tikzpicture}
    \caption{Two examples of multiplication on partially labeled graphs in the gluing algebra are illustrated for graphs $G_1$ and $G_2$ in (a) and graphs $H_1$ and $H_2$ in (b).} \label{fig:gluing}
\end{figure}
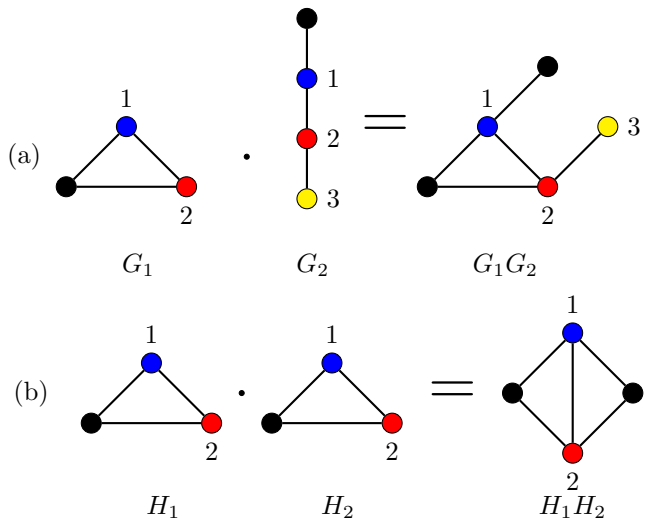

In this paper, we are particularly interested in graphs $G$ such that there exists a partially labeled graph $H$ such that $G$ is isomorphic to $HH$, and $H$ is not isomorphic to $G$. We call $G$ the \textit{square} of $H$ and $H$ a \textit{square root} of $G$, and an involution demonstrating $G \simeq HH$ a ``butterfly involution" (see Observation \ref{obs:alt-defn}). The graph $H_1H_2$ illustrated in Figure~\ref{fig:gluing}(b) is an example of a graph with a square root $H_1$ (or $H_2$), since $H_1$ and $H_2$ are isomorphic. In general, if a square root of a graph $G$ exists, it is not necessarily unique.

Hearkening back to our butterfly analogy in Section~\ref{sec:intro}, we give two possible square roots of the graph $B$ in Figure~\ref{fig:butterfly}. We note that this graph has many symmetries, and hence many nonisomorphic square roots. We give an example to demonstrate two of the graph's square roots.

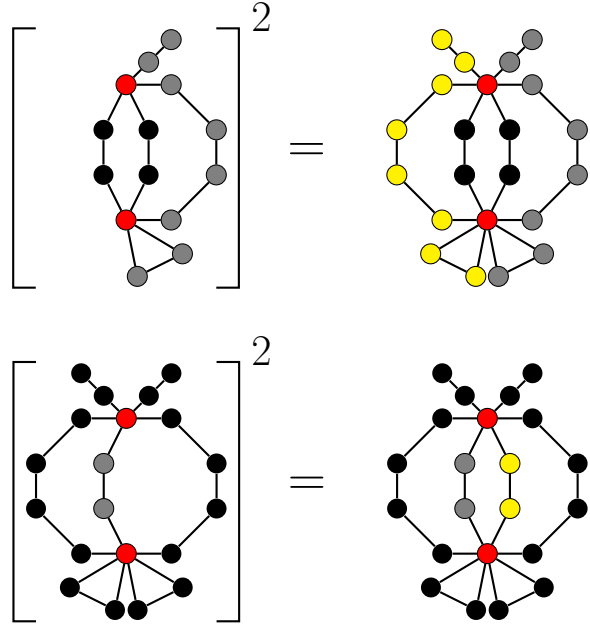
\begin{figure}[htb]
    \centering
    \begin{tikzpicture}[scale=0.6, every node/.style={scale=0.8}]
        \node[circle,fill=red,draw=black] (1) at (0,0) {};
        \node[circle,fill=gray,draw=black] (2) at (0.5,0.5) {};
        \node[circle,fill=gray,draw=black] (3) at (1,1) {};
        \node[circle,fill=black] (6) at (0.5,-1) {};
        \node[circle,fill=black] (7) at (0.5,-2) {};
        \node[circle,fill=red,draw=black] (8) at (0,-3) {};
        \node[circle,fill=black] (9) at (-0.5,-1) {};
        \node[circle,fill=black] (10) at (-0.5,-2) {};
        \node[circle,fill=gray,draw=black] (11) at (1,0) {};
        \node[circle,fill=gray,draw=black] (12) at (2,-1) {};
        \node[circle,fill=gray,draw=black] (13) at (2,-2) {};
        \node[circle,fill=gray,draw=black] (14) at (1,-3) {};
        \node[circle,fill=gray,draw=black] (15) at (1.25,-3.75) {};
        \node[circle,fill=gray,draw=black] (16) at (0.25,-4.25) {};


        \draw[black,thick] (3) -- (2) -- (1);
        \draw[black,thick] (1) -- (6) -- (7) -- (8) -- (10) -- (9) -- (1);
        \draw[black,thick] (1) -- (11) -- (12) -- (13) -- (14) -- (8) -- (15) -- (16) -- (8);

        \draw[black,thick] (-2,1.25) -- (-2.5,1.25) -- (-2.5,-4.5) -- (-2,-4.5);
        \draw[black,thick] (2,1.25) -- (2.5,1.25) -- (2.5,-4.5) -- (2,-4.5);
        \node at (3,1.5) {\huge $2$};

        \node at (4,-1.5) {\Huge $=$};

        \begin{scope}[shift = {(8, 0)}]
        \node[draw=black,circle,fill=red] (1) at (0,0) {};
        \node[draw=black,circle,fill=gray] (2) at (0.5,0.5) {};
        \node[draw=black,circle,fill=gray] (3) at (1,1) {};
        \node[draw=black,circle,fill=yellow] (4) at (-0.5,0.5) {};
        \node[draw=black,circle,fill=yellow] (5) at (-1,1) {};
        \node[draw=black,circle,fill=black] (6) at (0.5,-1) {};
        \node[draw=black,circle,fill=black] (7) at (0.5,-2) {};
        \node[draw=black,circle,fill=red] (8) at (0,-3) {};
        \node[draw=black,circle,fill=black] (9) at (-0.5,-1) {};
        \node[draw=black,circle,fill=black] (10) at (-0.5,-2) {};
        \node[draw=black,circle,fill=gray] (11) at (1,0) {};
        \node[draw=black,circle,fill=gray] (12) at (2,-1) {};
        \node[draw=black,circle,fill=gray] (13) at (2,-2) {};
        \node[draw=black,circle,fill=gray] (14) at (1,-3) {};
        \node[draw=black,circle,fill=gray] (15) at (1.25,-3.75) {};
        \node[draw=black,circle,fill=gray] (16) at (0.25,-4.25) {};

        \node[draw=black,circle,fill=yellow] (17) at (-1,0) {};
        \node[draw=black,circle,fill=yellow] (18) at (-2,-1) {};
        \node[draw=black,circle,fill=yellow] (19) at (-2,-2) {};
        \node[draw=black,circle,fill=yellow] (20) at (-1,-3) {};
        \node[draw=black,circle,fill=yellow] (21) at (-1.25,-3.75) {};
        \node[draw=black,circle,fill=yellow] (22) at (-0.25,-4.25) {};

        \draw[black,thick] (3) -- (2) -- (1) -- (4) -- (5);
        \draw[black,thick] (1) -- (6) -- (7) -- (8) -- (10) -- (9) -- (1);
        \draw[black,thick] (1) -- (11) -- (12) -- (13) -- (14) -- (8) -- (15) -- (16) -- (8);
        \draw[black,thick] (1) -- (17) -- (18) -- (19) -- (20) -- (8) -- (21) -- (22) -- (8);
        \end{scope}
    \end{tikzpicture}

    \vspace{.2in}

    \begin{tikzpicture}[scale=0.6, every node/.style={scale=0.8}]
        \node[circle,fill=red,draw=black] (1) at (0,0) {};
        \node[circle,fill=black] (2) at (0.5,0.5) {};
        \node[circle,fill=black] (3) at (1,1) {};
        \node[circle,fill=black] (4) at (-0.5,0.5) {};
        \node[circle,fill=black] (5) at (-1,1) {};
        \node[circle,fill=red,draw=black] (8) at (0,-3) {};
        \node[circle,fill=gray,draw=black] (9) at (-0.5,-1) {};
        \node[circle,fill=gray,draw=black] (10) at (-0.5,-2) {};
        \node[circle,fill=black] (11) at (1,0) {};
        \node[circle,fill=black] (12) at (2,-1) {};
        \node[circle,fill=black] (13) at (2,-2) {};
        \node[circle,fill=black] (14) at (1,-3) {};
        \node[circle,fill=black] (15) at (1.25,-3.75) {};
        \node[circle,fill=black] (16) at (0.25,-4.25) {};

        \node[circle,fill=black] (17) at (-1,0) {};
        \node[circle,fill=black] (18) at (-2,-1) {};
        \node[circle,fill=black] (19) at (-2,-2) {};
        \node[circle,fill=black] (20) at (-1,-3) {};
        \node[circle,fill=black] (21) at (-1.25,-3.75) {};
        \node[circle,fill=black] (22) at (-0.25,-4.25) {};

        \draw[black,thick] (3) -- (2) -- (1) -- (4) -- (5);
        \draw[black,thick] (8) -- (10) -- (9) -- (1);
        \draw[black,thick] (1) -- (11) -- (12) -- (13) -- (14) -- (8) -- (15) -- (16) -- (8);
        \draw[black,thick] (1) -- (17) -- (18) -- (19) -- (20) -- (8) -- (21) -- (22) -- (8);

        \draw[black,thick] (-2,1.25) -- (-2.5,1.25) -- (-2.5,-4.5) -- (-2,-4.5);
        \draw[black,thick] (2,1.25) -- (2.5,1.25) -- (2.5,-4.5) -- (2,-4.5);
        \node at (3,1.5) {\huge $2$};

        \node at (4,-1.5) {\Huge $=$};

        \begin{scope}[shift = {(8, 0)}]
        \node[circle,fill=red,draw=black] (1) at (0,0) {};
        \node[circle,fill=black] (2) at (0.5,0.5) {};
        \node[circle,fill=black] (3) at (1,1) {};
        \node[circle,fill=black] (4) at (-0.5,0.5) {};
        \node[circle,fill=black] (5) at (-1,1) {};
        \node[circle,fill=yellow,draw=black] (6) at (0.5,-1) {};
        \node[circle,fill=yellow,draw=black] (7) at (0.5,-2) {};
        \node[circle,fill=red,draw=black] (8) at (0,-3) {};
        \node[circle,fill=gray,draw=black] (9) at (-0.5,-1) {};
        \node[circle,fill=gray,draw=black] (10) at (-0.5,-2) {};
        \node[circle,fill=black] (11) at (1,0) {};
        \node[circle,fill=black] (12) at (2,-1) {};
        \node[circle,fill=black] (13) at (2,-2) {};
        \node[circle,fill=black] (14) at (1,-3) {};
        \node[circle,fill=black] (15) at (1.25,-3.75) {};
        \node[circle,fill=black] (16) at (0.25,-4.25) {};

        \node[circle,fill=black] (17) at (-1,0) {};
        \node[circle,fill=black] (18) at (-2,-1) {};
        \node[circle,fill=black] (19) at (-2,-2) {};
        \node[circle,fill=black] (20) at (-1,-3) {};
        \node[circle,fill=black] (21) at (-1.25,-3.75) {};
        \node[circle,fill=black] (22) at (-0.25,-4.25) {};

        \draw[black,thick] (3) -- (2) -- (1) -- (4) -- (5);
        \draw[black,thick] (1) -- (6) -- (7) -- (8) -- (10) -- (9) -- (1);
        \draw[black,thick] (1) -- (11) -- (12) -- (13) -- (14) -- (8) -- (15) -- (16) -- (8);
        \draw[black,thick] (1) -- (17) -- (18) -- (19) -- (20) -- (8) -- (21) -- (22) -- (8);
        \end{scope}
    \end{tikzpicture}
    \caption{Two of the many butterfly involutions of the graph $B$ in Figure~\ref{fig:butterfly}. Black and red vertices correspond to labeled vertices and gray vertices are unlabeled before the squaring operation. After the squaring operation, the ``new'' vertices are in yellow.} \label{fig:butterfly-root}
\end{figure}

\begin{example} \rm \label{exa:butterfly} 
In this example, we demonstrate two ways to realize the graph $B$ in Figure~\ref{fig:butterfly} as a square graph. The first example is to label the central $C_6$ containing the blue and red vertices. The remaining vertices on the right each pair with the corresponding vertex on the left via reflection through the vertical line through the blue and red vertices. The second example is to label all the vertices except the two vertices on the right hand path of length 3 between the blue and red vertices. The two unlabeled vertices pair with the corresponding vertices on the left. These two examples are illustrated in Figure~\ref{fig:butterfly-root}.
\end{example}

We now discuss some alternate definitions of a square, and use them to prove additional results. The following is a restatement of \cite[Lemma 2.8]{BRST20}, 
\begin{observation} \label{obs:alt-defn} 
An alternative definition is that $G$ is a square  if and only if there is an automorphism $\phi$ of $G$, with the set of vertices fixed by $\phi$ denoted by $F_{\phi}(G)$, such that 
\begin{enumerate} 
\item $F_{\phi}(G)\neq \emptyset$
\item $V(G)-F_{\phi}(G)\neq \emptyset$
\item if $u\in V(G)-F_{\phi}(G)$, then $N(u)\cap N(\phi(u))=N(u)\cap F_{\phi}(G)$, and 
\item there is a partition $A_0\cup A_1$ of $V(G)-F_{\phi}(G)$ such that there are no edges with one vertex in $A_0$ and one vertex in $A_1$, and if $u\in A_k$, then $\phi(u) \in A_{1-k}$ for $0 \leq k \leq 1$.
\end{enumerate} 

We refer to any such $\phi$ as a \textit{butterfly involution} of $G$.  
\end{observation}

Note that when $G=HH$, then $F_{\phi}(G)$ is the set of labeled vertices of $H$. 

\begin{definition}[\cite{AAP}] \label{def:original} \rm Let $H$ be a partially labeled graph with map $\theta:L\rightarrow V(H)$. We write $V(H)=\{v_1, v_2, \ldots, v_k\}\cup \theta(L)$, and
$$
V(HH)=\{v_1, v_2, \ldots, v_k\} \cup \{v'_1, v'_2, \ldots, v'_k\}\cup\theta(L),
$$ 
where $v_i$ and $v'_i$ correspond to the same unlabeled vertex in $H$ for $1\leq i\leq k$. We say that $v,v'$ are \emph{twins}.
\end{definition} 
The above is used in Lemma 2.1 of \cite{AAP}.  We recall the next definition. 

\begin{definition} \rm Let $G=(V,E)$ and $v\in V(G)$. The set $N(v)=\{u:uv \in E\}$ is the \emph{open neighborhood} of $v$ in $G$, and the set $N[v]=N(v) \cup \{v\}$ is the \emph{closed neighborhood} of $v$ in $G$.
\end{definition} 

We now make some observations about squares and square roots of graphs that will be useful in the sequel.

\begin{observation} \rm \label{obs:connected} 
If $H$ is not connected, then $HH$ is not connected. 
\end{observation}
\noindent Indeed, suppose $H$ has components $U_1, \ldots, U_\ell$ and $L_{\ell+1}, \ldots, L_k$ with none of the $U_i$'s having any labeled vertices and all of the $L_j$'s having some labeled vertices. Then $HH$ has components $U_1, \ldots, U_\ell, U'_1, \ldots, U'_\ell, L_{\ell+1}L_{\ell+1}, \ldots, L_kL_k$, where $U'_i$ is isomorphic to $U_i$ for $i=1, \ldots, \ell$. In particular, if $H$ is unlabeled then $HH$ has at least two components and so is not connected. 

\begin{observation} \rm \label{obs:subgraph} 
If $G$ is isomorphic to $HH$, then $H$ is isomorphic to a vertex-induced subgraph of of $G$ (in fact, more than one).  If $\varphi$ is an isomorphism from $H$ to $G$ then every vertex $x\in V(G)-\varphi(V(H)$) is a twin of a vertex in $\varphi(V(H))$.
\end{observation} 
\noindent This is evident from the definition of square graph.

\begin{observation} \rm \label{obs:two-times} 
If $G$ is the square of $H$ then $|V(G)| \leq 2|V(H)|$ and $|E(G)| \leq 2|E(H)|$.
\end{observation}
\noindent Indeed, $|V(G)|=2|V(H)|-|L|\leq 2|V(H)|$ and 
$$
|E(G)|= 2|E(H)|-|E(H[\theta(L)])|\leq 2|E(H)|,
$$ 
where here we use $H[S]$ to denote the subgraph of $H$ induced by the set of vertices $S$. 

\begin{observation} \rm \label{obs:deg} \rm 
Let $H$ be partially labeled. If $v$ is an unlabeled vertex, then in $HH$, $v$ and its twin $v'$ satisfy $\deg_{HH}(v)=\deg_{HH}(v')=\deg_H(v)$. Further, if $i$ is a labeled vertex with $\alpha$ labeled vertices in $N_H(i)$, then $\deg_{HH}(i)=2\deg_H(i)-\alpha$. 
\end{observation} 
\noindent For the first equality note that in $HH$ $v$ has the same number of neighbors among $\{v_1, \ldots, v_k\}$ as $v'$ has among $\{v'_1, \ldots, v_k\}$ (here using the notation of Definition \ref{def:original}), and that $v$ and $v'$ both have the same number of neighbors among $\theta(L)$. For the second equality note that in $HH$, $i$ has $\deg_H(i)-\alpha$ neighbors among each of $\{v_1, \ldots, v_k\}$ and $\{v'_1, \ldots, v'_k\}$, and $\alpha$ neighbors among $\theta(L)$, so $2(\deg_H(i)-\alpha)+\alpha = 2\deg_H(i)-\alpha$ neighbors in all.      

\begin{observation} \rm \label{obs:twin} 
Let $H$ be partially labeled. If $v$ is an unlabeled vertex in $V(H)$ then twins $v$ and $v'$ are nonadjacent in $HH$, and $N_{HH}(v)\cap N_{HH}(v')$ is the subset of labeled vertices in $N_H(v)$. Further, if $v_1, v_2$ are unlabeled adjacent vertices in $H$, then their twins $v'_1,v'_2$ are unlabeled adjacent vertices in $HH$.
\end{observation}
\noindent For this see \cite[Lemma 2.1 and ensuing discussion]{AAP}.

\begin{observation} \rm \label{obs:dominating}
    If $v \in G$ is connected to all other vertices in $G$, then $v \in F_\phi(G)$. This is because if $v \in A_i$, it must have some twin in $A_{1-i}$ for $i \in \{0,1\}$ that it is not adjacent to. 
\end{observation}
\noindent This follows from Observation~\ref{obs:twin} and also from Observation \ref{obs:connected}.

\begin{observation} \rm \label{obs:chromatic_critical}
If a graph $H$ admits a $k$-coloring, then so does $HH$. This means that the chromatic number of a square root of $G$ cannot be smaller than the chromatic number of $G$, and in particular a vertex chromatic critical graph cannot be a square.     
\end{observation}
\noindent Indeed, given a $k$-coloring of $H$ on vertex set $\{v_1, \ldots, v_k\} \cup \theta(L)$ we extend to a $k$-coloring of $HH$ by giving $v_i'$ the same color as $v_i$ for each $i$.

\section{General Results} \label{sec:general}

\begin{proposition} \label{prop:smaller-L} Let $H$ be a partially labeled graph, with labeling $\theta:L\rightarrow V(H)$. Let $T$ be a proper subset of $L$. The subgraph of $HH$ induced by $V(HH)-\theta(T)$ is the square of the subgraph of $H$ induced by $V(H)-\theta(T)$; in other words, $HH-\theta(T)=(H-\theta(T))(H-\theta(T))$.
\end{proposition}

\begin{proof}
Let $\phi:V(HH)\rightarrow V(HH)$ be a butterfly involution. 
Then $\phi$ restricted to $V(HH)-\theta(T)$ is an involution that fixes $V(H)-\theta(T)$ and that satisfies $(1), (2), (3), (4)$.
\end{proof}

In addition to determining whether a particular graph is a square, we are also interested in constructing squares. The next theorem shows how to start with a graph $G$ that has a non-trivial involution and delete edges to create a square subgraph of $G$. An example is given in Figure~\ref{fig:circ-minus-e}.

\begin{theorem} \label{prop:subgraph-e} 
Let graph $G$ have an automorphism $\phi:V(G)\rightarrow V(G)$ which is an involution. Let $F_{\phi}(G)$ be the set of vertices fixed by $\phi$. Choose a partition $A_0\cup A_1$ of $ V(G)- F_{\phi}(G)$ so that for each pair $u, \phi(u)\in V(G)- F_{\phi}(G)$, either $u\in A_0$ and $\phi(u)\in A_1$, or 
$u\in A_1$ and $\phi(u)\in A_0$. Let the set of edges $D=\{\{u,v\}\;|\; u\in A_0, v\in A_1\}$. If $F_{\phi}(G)\neq \emptyset$ and  $V(G)-F_{\phi}(G)\neq \emptyset$, then $G'=(V(G), E(G)-D)$ is a square.
\end{theorem}

\begin{proof}
We will use Observation~\ref{obs:alt-defn} to show $G'$ is a square. Since $\phi$ is an automorphism and an involution, the restriction $\phi'$ to $G'$ is an involution of  $V(G')=V(G)$ as a set. We want to show that it is an automorphism of $G'$, that is, we want to show $w_1w_2\in E(G')$ if and only if $\phi(w_1)\phi(w_2)\in E(G')$. If $w_1, w_2$ are adjacent in $G'$, without loss of generality, then 
$w_1, w_2\in F_{\phi}(G)$, or $w_1\in A_k$ and $w_2\in F_{\phi}(G)$, or $w_1, w_2\in A_k$ for $k=0$ or $1$. 

Suppose $w_1, w_2\in F_{\phi}(G)$, or $w_1\in A_k$ and $w_2\in F_{\phi}(G)$. Since $D$ does not contain edges with an endpoint in $F_{\phi}(G)$, the edge $w_1w_2\in E(G')$ if and only if the edge $w_1w_2\in E(G)$. Since $\phi$ is an automorphism, $w_1w_2\in E(G)$ if and only if $\phi(w_1)\phi(w_2)\in E(G)$, hence 
$w_1w_2\in E(G')$ if and only if $\phi(w_1)\phi(w_2)\in E(G')$.

Suppose $w_1, w_2\in A_k$ for $k=0$ or $1$. Since $D$ does not contain edges between two vertices both in $A_k$, then the edge $w_1w_2\in E(G')$ if and only if the edge $w_1w_2\in E(G)$. Since $\phi$ is an automorphism, $w_1w_2\in E(G)$ if and only if $\phi(w_1)\phi(w_2)\in E(G)$. By hypothesis, $\phi(w_1), \phi(w_2)\in A_{1-k}$, and the edge $\phi(w_1)\phi(w_2)\not \in D$, hence $\phi(w_1)\phi(w_2)\in E(G)$ if and only if $\phi(w_1)\phi(w_2)\in E(G')$. 
Therefore, $w_1w_2\in E(G')$ if and only if $\phi(w_1)\phi(w_2)\in E(G')$.

By hypothesis, conditions (1), (2) and (4) of Observation~\ref{obs:alt-defn} are met. To show (3) holds, we observe that if $u\in V(G)-F_{\phi}(G)$, and $u\in A_k$, then $\phi(u)\in A_{1-k}$, for $0\leq k\leq 1$, and, since $D$ contains all edges between $A_0$ and $A_1$, then $N(u)\cap N(\phi(u))\subseteq F_{\phi}(G)$. 
Hence $G'$ is a square. 
\end{proof}

We conclude this section with an additional characterization of square graphs that we originally developed for work on trees, but is in fact valid for all graphs.

\begin{proposition}\label{prop:2components}
    If a connected graph $G$ contains a cut-set $S$ so that its removal creates (possibly among others) two components $H_1$ and $H_2$ such that the subgraphs $H_1'$ (resp. $H_2'$) induced by $H_1 \cup S$ (resp. $H_2 \cup S$) admit an isomorphism $\varphi: H_1'\to H_2'$  that preserves $S$, then $G$ is a square.  Furthermore any square contains such a cut-set.
\end{proposition}
\begin{proof}
    Let $G$ be such a graph and construct a square root as follows: label $G \setminus \{H_1 \cup H_2\}$, remove $H_2$, and leave $H_1$ unlabeled.  If $G$ is such a square, then the cut-set is the labeled vertices.
\end{proof}

Figure~\ref{fig:not-squ} shows that all hypotheses are necessary in Proposition~\ref{prop:2components}. 

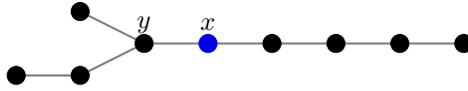
\begin{figure}[ht] 
\begin{center} 
\begin{tikzpicture}[scale=.85]
\draw[thick, gray] (0,.5)--(1,0)--(6,0);
\draw[thick, gray] (-1,-.5)--(0,-.5)--(1,0);

\filldraw[blue!90!black]
(2,0) circle [radius=4pt]
;
\filldraw[black]
(0,.5) circle [radius=4pt]
(0,-.5) circle [radius=4pt]
(-1,-.5) circle [radius=4pt]
(1,0) circle [radius=4pt]
(3,0) circle [radius=4pt]
(4,0) circle [radius=4pt]
(5,0) circle [radius=4pt]
(6,0) circle [radius=4pt]
;
\node(0) at (2,.3) {$x$};
\node(0) at (1,.3) {$y$};
\end{tikzpicture}

\caption{Both components in $G-x$ are isomorphic to $P_4$, but $G$ is not a square.}
\label{fig:not-squ}

\end{center}
\end{figure}

\begin{corollary} \label{cor:twins}
Let $G$ be a graph with at least 3 vertices and such that $v_1,v_2$ have the same open neighborhood. Then $G$ is a square. 
\end{corollary}
\begin{proof}
$G$ has a cut-set $S=V(G)-\{v_1, v_2\}$, whose removal creates a graph with two components $H_1=\{v_1\}$ and $H_2=\{v_2\}$. Since the open neighborhoods of $v_1$ and $v_2$ are equal, by Proposition~\ref{prop:2components}, $G$ is a square.
\end{proof}
The above corollary is in \cite{AAP} as Lemma 2.1; we include it here as a nice application of Proposition \ref{prop:2components}.

\section{Families of Square and Non-Square Graphs} \label{sec:families}

In this section, we consider some specific families of graphs and show that they are either squares or non-squares.

\subsection{Basic Graph Families}
In this subsection, we consider cycles, paths, complete graphs, and wheels, in that order. Throughout, $C_n$ is the cycle with $n$ vertices and $P_n$ is the path with $n$ vertices.  Our proofs use both Definition~\ref{def:original} and associated observations as well as using the alternate definition of Observation~\ref{obs:alt-defn}. We do this to illustrate the utility of both definitions.  We make use of one or the other as makes sense in the given setting

We begin by characterizing when the cycle is a square. Though this seems to be well-known, we believe this is the first printed proof.

\begin{proposition} \label{pr:cycles}
    $C_n$ is a square if and only if $n$ is even. Further, for $n$ even, there is a unique way to write $C_n$ as a square, up to labeling.
\end{proposition}

\begin{proof}
Let $C_n$ be isomorphic to $HH$. Given that $C_{n}$ is the square of $H$, then by Observations~\ref{obs:connected} and \ref{obs:subgraph}, $H$ is isomorphic to a connected subgraph of $C_n$, and therefore $H$ is isomorphic to a path $P_{m+1}$, and by Observation~\ref{obs:two-times}, we have $m\geq \lceil\frac{n}{2}\rceil$. Since each vertex in $C_n$ has degree 2, by Observation~\ref{obs:deg}, each vertex of degree 1 in $H$ is labeled, and for each labeled vertex $i$ of degree 2 with $\alpha$ labeled neighbors, we have $2=2\deg_H(i)-\alpha$, and thus $\alpha=2$. Hence both vertices in the neighborhood of each labeled vertex of degree 2 are labeled in $H$. If there is a labeled vertex of degree 2, by induction, every vertex in $H$ is labeled. Since by definition not every vertex in $H$ is labeled, only the endpoints of $P_{m+1}$ are labeled. By identification, $HH$ is isomorphic to $C_{2m}$. Thus, the square root of $C_{2m}$ is unique up to the choice of labels of the endpoints of $P_{m+1}$ and $C_n$ is not a square for $n$ odd.  

For a second approach, first consider $C_n$ for $n$ even and fix two antipodal vertices and reflect the cycle around the axis of symmetry containing them. This is an involution that satisfies the conditions of Observation~\ref{obs:alt-defn}. For $n$ odd, the only involutions with a fixed point $x$ are the reflection $\phi$  around the axis of symmetry containing $x$ and the midpoint of the edge $yz$ opposite to $x$ in the cycle. Then $\phi(y)=z$, but $yz\in E(C_n)$, which is a contradiction, and shows $C_n$ is not a square when $n$ is odd.

An alternate proof that $C_n$ is not a square when $n$ is odd follows directly from Observation \ref{obs:chromatic_critical}: any odd cycle is vertex chromatic critical.

\end{proof}
The following is Corollary 1.3 of \cite{BRST20}, but proved here using our new machinery.
\begin{proposition}\cite{BRST20} \label{prop:paths}
    $P_{2k+1}$ is a square, and $P_{2k}$ is a non-square. Further, there is a unique way to write $P_{2k+1}$ as a square, up to labeling.
\end{proposition}

\begin{proof}
    Let $v$ and $w$ be the end vertices of $P_n$. If $\phi$ is any graph automorphism of $P_n$, it must be that $\phi(v) = v$ and $\phi(w) = w$ (corresponding to the identity morphism) or $\phi(v) = w$ and $\phi(w) = v$ (corresponding to the automorphism that reverses the order of vertices). If $n=2k$ is even, there are no fixed points of this involution, and so property (1) of Observation~\ref{obs:alt-defn} cannot hold, and $P_{2k}$ cannot be a square. If $n=2k+1$ is odd, there is exactly one fixed point of this involution, the middle vertex. It is easily verifiable that all properties of Observation~\ref{obs:alt-defn} hold and so $P_{2k+1}$ has a unique square root.
\end{proof}

We now move on to complete graphs. Proposition \ref{pr:complete} below is straightforward, but we include the proof for completeness as we are not aware of this having been written down previously. Note that this is a nice complement to Lemma 2.6 in \cite{AAP}, which characterizes subdivisions of complete graphs. 

\begin{proposition} \label{pr:complete}
    $K_n$ is not a square for $n \geq 1$.
\end{proposition}

\begin{proof}
    Suppose that $K_n$ is isomorphic to $HH$. By definition, not every vertex in $H$ is labeled. Let $v\in V(H)$ be unlabeled. By Observation~\ref{obs:twin}, then $v$ and $v'$ are nonadjacent vertices in $HH$, and since every pair of vertices is adjacent in $K_n$, this contradicts the fact that $K_n$ is the square of $H$. 

An alternate proof uses Observation \ref{obs:chromatic_critical}: complete graphs are vertex chromatic critical and therefore not squares.
\end{proof}

In contrast, all complete multipartite graphs that are not $K_n$ are squares. This is a straightforward corollary of Corollary~\ref{cor:twins}, i.e. Lemma 2.1 in \cite{AAP}. 

\begin{proposition} \label{pr:multipartite}
    Let $G = K_{a_1, a_2, \dotsc, a_k}$ be a complete multipartite graph with $k \geq 2$. If any $a_i \geq 2$, then $G$ is a square.
\end{proposition}

\begin{proof}
    Let $A_i$ be a partite set of $G$ such that $|A_i| \geq 2$. Pick two vertices $v_1, v_2 \in A_i$. $N(v_1) = N(v_2)$ and is nonempty as $k \geq 2$. By Corollary~\ref{cor:twins}, $G$ is a square. 
\end{proof}

Recall that the wheel $W_n$ is $C_n$ with an additional vertex $w$ that is adjacent to every vertex in $C_n$. Every vertex in $W_n$ has degree 3, except $w$, which has degree $n$, and $|V(W_n)|=n+1$, $|E(W_n)|=2n$.

\begin{proposition} \label{pr:wheels}
$W_n$ is a square for $n$ even and is not a square for $n$ odd. Further, the square root of $W_n$ is unique up to isomorphism. 
\end{proposition}

\begin{proof}
Since $W_3$ is isomorphic to $K_4$, by Proposition~\ref{pr:complete}, $W_3$ is not a square. We prove the remaining cases by analyzing the details of any hypothetical square root of $W_n$ for $n \geq 4$.  The analysis shows that its existence and structure is fully determined by the parity of $n$.  (An alternate approach to showing $W_n$ not a square for odd $n$ is to observe such a graph is vertex chromatic critical and apply Observation \ref{obs:chromatic_critical}; however this offers no insight into the even case.)

Suppose that $H$ is a partially labeled graph and $W_n$ is isomorphic to $HH$. By Observations~\ref{obs:connected} and \ref{obs:subgraph}, $H$ is a connected, vertex-induced subgraph of $W_n$. Let $v\in V(H)$ be unlabeled. Then by Observation~\ref{obs:twin}, twins $v,v'$ are nonadjacent in $HH$, and since $w$ is a 
universal vertex, $v, v'\neq w$. Also by Observation~\ref{obs:twin}, $N_{HH}(v)\cap N_{HH}(v)$ is the subset of labeled vertices in $N_H(v)$. Since $w$ is a universal vertex, $w\in N_{HH}(v)\cap N_{HH}(v)$, and thus $w$ is labeled in $H$. Now the unlabeled vertices in $W_n$ that are on $C_n$ are paired. If $n$ is odd, there is another labeled vertex in $H$. If $n$ is even, there are at most $\frac{n}{2}$ unlabeled vertices in $H$, and the unlabeled vertices together 
with $w$ induce a subgraph of $W_n$ with at most $(\frac{n}{2}-1)+\frac{n}{2}=n-1$ edges; by Observation~\ref{obs:two-times}, $|E(HH)|\leq 2n-2$, which contradicts the fact that $|E(W_n)|=2n$. Hence, there is another labeled vertex in $H$ besides $w$. Let the vertices of $C_n\subseteq W_n$ be $\{x_0,x_1, x_2, \ldots, x_{n-1}\}$ with edges between consecutive pairs modulo $n$. Since there are both unlabeled and labeled vertices in $H$ that are in $C_n$, without loss of generality, let $x_0$ be a labeled vertex and its neighbor $x_1$ be unlabeled in $H$. By 
Observation~\ref{obs:twin} and since $x_0x_1\in E(W_n)$, $x_{n-1}=x_1'$. Since $N_{W_n}(x_1)\cap N_{W_n}(x_{n-1})=\{w, x_0\}$, then by Observation~\ref{obs:twin}, $x_2$ is unlabeled in $H$, and also $x_{n-2}=x'_2$. By induction, $x_1, x_2, \cdots, x_k$ are unlabeled in $H$ and $x_{n-k}=x'_k$ for $1\leq j\leq k$, as long as $N_{W_n}(x_k)\cap N_{W_n}
(x_{n-k})=\{w\}$. When $n$ is odd,  $x_1, x_2, \cdots, x_{\frac{n-1}{2}}$ are unlabeled in $H$ and $x_{n-k}=x'_k$ for $1\leq j\leq \frac{n-1}{2}$. Since $n-\frac{n-1}{2}=\frac{n+1}{2}$, then $x_{\frac{n+1}{2}}=x'_{\frac{n-1}{2}}$. However, $x_{\frac{n-1}{2}}x_{\frac{n+1}{2}}\in E(W_n)$, this contradicts the fact that twins are nonadjacent. Thus, $W_n$ is not a square for $n$ odd. When $n$ is even,  $x_1, x_2, \cdots, x_{\frac{n-2}{2}}$ are unlabeled in $H$ and 
$x_{n-k}=x'_k$ for $1\leq j\leq \frac{n-2}{2}$, and $x_{\frac{n}{2}}$ is unpaired, and therefore must be labeled in $H$. Thus, in this case, $HH$ is isomorphic to $W_n$, and $W_n$ is a square for $n$ even, with a square root that is unique up to isomorphism. See Figure~\ref{fig:even-wheel}.
\end{proof}

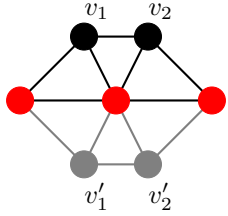
\begin{figure}[ht] 
\begin{center} 
\begin{tikzpicture}[scale=.85] 
\draw[thick, gray] (0,1)--(1,0)--(2,0)--(3,1);
\draw[thick, black] (3,1)--(2,2)--(1,2)--(0,1);
\draw[thick, gray] (1.5,1)--(1,0);
\draw[thick, gray] (1.5,1)--(2,0);
\draw[thick, black] (1.5,1)--(0,1);
\draw[thick, black] (1.5,1)--(3,1);
\draw[thick, black] (1.5,1)--(1,2);
\draw[thick, black] (1.5,1)--(2,2);

\filldraw[red]
(1.5,1) circle [radius=6pt]
(0,1) circle [radius=6pt]
(3,1) circle [radius=6pt]
;
\filldraw[black]
(1,2) circle [radius=6pt]
(2,2) circle [radius=6pt]
;
\filldraw[gray]

(1,0) circle [radius=6pt]
(2,0) circle [radius=6pt]
;
\node(0) at (1.2,2.4) {$v_1$};
\node(0) at (1.2,-.5) {$v'_1$};
\node(0) at (2.2,2.4) {$v_2$};
\node(0) at (2.2,-.5) {$v'_2$};
\end{tikzpicture}

\caption{A butterfly involution of the even wheel $W_6$ is illustrated. Vertices $v_1$ and $v'_1$ are twinned, as are $v_2$ and $v'_2$.}
\label{fig:even-wheel}

\end{center}
\end{figure}

\subsection{Circulant Graphs}

The next class of graphs we consider are circulant graphs. We begin with a formal definition.

\begin{definition}
    \rm Let $1\leq d_1<d_2< \cdots <d_s\leq \frac{n}{2}$. 
    The circulant graph $C_n(d_1,d_2,\dots,d_s)$ is the graph with vertex set $\{0,1,2,3, \ldots, n-1\}$ and $ij$ is an edge if $i - j \equiv \pm \; d_k \pmod n$  for some $k$ with $1 \le k \le s$. Given $S=\{d_1, d_2, \ldots, d_s\}$, we write $C_n(S)$ for $C_n(d_1,d_2,\dots,d_s)$.
\end{definition}

\begin{example} \rm 
    To illustrate this definition, we provide the example circulant graph $C_{10}^{1,4}$ in Figure~\ref{fig:circulant}. 
\end{example}

\begin{figure}[ht] 
\centering
\begin{tikzpicture}[scale=.65]
\node[circle,fill=black,label=above:{$0$}] (0) at (0,0) {};
\node[circle,fill=black,label=left:{$1$}] (1) at (-1,-0.75) {};
\node[circle,fill=black,label=left:{$2$}] (2) at (-1.5,-2) {};
\node[circle,fill=black,label=left:{$3$}] (3) at (-1,-3.25) {};
\node[circle,fill=black,label=below:{$4$}] (4) at (0,-4) {};
\node[circle,fill=black,label=below:{$5$}] (5) at (1,-4) {};
\node[circle,fill=black,label=right:{$6$}] (6) at (2,-3.25) {};
\node[circle,fill=black,label=right:{$7$}] (7) at (2.5,-2) {};
\node[circle,fill=black,label=right:{$8$}] (8) at (2,-0.75) {};
\node[circle,fill=black,label=above:{$9$}] (9) at (1,0) {};

\draw[black,thick] (0) -- (1) -- (2) -- (3) -- (4) -- (5) -- (6) -- (7) -- (8) -- (9) -- (0);
\draw[black,thick] (0) -- (4) -- (8) -- (2) -- (6) -- (0);
\draw[black,thick] (1) -- (5) -- (9) -- (3) -- (7) -- (1);
\end{tikzpicture}

\caption{The circulant graph $C_{10}^{1,4}$.}
\label{fig:circulant}
\end{figure}
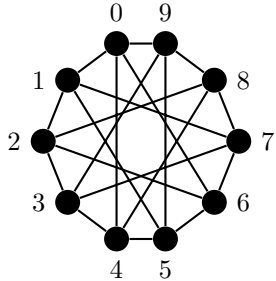

Circulants are regular and their degree is determined by the set $S$ of differences: a circulant $C_n^{d_1, \dotsc, d_k}$ has degree $2k$ if $d_k \neq \frac{n}{2}$ and degree $2k-1$ if $d_k = \frac{n}{2}$.

We begin with a lemma that will be used in some proofs in this section.

\begin{lemma} \label{lemma:connect} 
 If $H$ is a partially labeled and connected graph such that $L\neq\emptyset$,  then there is a labeled vertex $i$ and an unlabeled vertex $v$ such that $i$ is adjacent to $v$. 
\end{lemma}

\begin{proof}
Let $j\in \theta(L)$. By definition, $H$ must contain an unlabeled vertex $u$. Since $H$ is connected, there is a path $P$ from $j$ to $u$. The first unlabeled vertex $v$ on $P$ is adjacent to the previous labeled vertex $i$.
 \end{proof}

The next theorem contrasts with the fact that $C_n$ is a square if $n$ is even. 

\begin{theorem}
If $k\geq 2,$ then $C_{n}^{1,2, \ldots k}$ is not a square. 
\end{theorem}

\begin{proof} 
For a proof by contradiction, suppose $C_{n}^{1,2, \ldots k}$ is a square. By Lemma~\ref{lemma:connect}, there is an edge between a labeled vertex and an unlabeled vertex. Without loss of generality, let the labeled vertex be 0 and the unlabeled vertex be 1. Let the twin of 1 be $x$. By Observation~\ref{obs:twin}, $x$, must be adjacent to 0. Since 1 and $x$ are not adjacent, 
$x\in N(0)\cap \overline{N(1)}$. Now $N(0)=\{-k,-(k-1),\ldots, -1, 1, 2, \ldots, k\}$ and $N(1)=\{-(k-1),\ldots, -1, 0,  2, \ldots, k+1\}$. Thus, $x=-k$. By Observation~\ref{obs:twin}, $N(1)\cap N(-k)=\{-(k-1), -(k-2), \ldots, 0\}$ are all labeled vertices, and neither 1 nor $-k$ are adjacent to any other labeled vertices. Since $1$ is adjacent to $2$, and $-k$ is not adjacent to 2, then 2 must 
be unlabeled. Let $y$ be 2's twin. Since 2 and $y$ have the same adjacencies among the labeled vertices, then both 2 and $y$ are adjacent to the labeled vertices  $\{-(k-2), -(k-3), \ldots, -1,0\}$. 
Now $N(-(k-2))\cap N(-(k-3)) \cap \ldots \cap N(-1)\cap N(0) = \{ -k, -(k-1),1,2\}$, but $y$ must be unlabeled, so $y\neq -k, -(k-1)$, and $y$ cannot be 1 because 1 and 2 are adjacent, and $y\neq 2$ because they are twins, so $y$ does not exist. This contradicts the assumption that $C_{n}^{1,2, \ldots k}$ is a square. 
\end{proof}

\begin{proposition}
    \label{prop:circ-D} Let $n\geq 4$ be even and $k\geq 2$. 
The circulant $C_{n}^{d_1,d_2, \ldots, d_k}$ has a subgraph with $n$ vertices and
\begin{enumerate}
    \item if $d_k \neq \frac{n}{2}$, has $nk - 2 \sum_{i=1}^k (d_i-1)$ edges and
    \item if $d_k = \frac{n}{2}$, has $n\left(k-\frac{1}{2}\right)-2\sum_{i=1}^{k-1}(d_i-1)$ edges
\end{enumerate}
that is a square.
\end{proposition}

\begin{proof} 
Fix two antipodal vertices $x_1, x_2$ of $C_{n}^{d_1,d_2, \ldots, d_k}$. The reflection around the axis of symmetry containing them is an automorphism, and an involution. Let $A_0$ be the set of vertices from $x_1$ to $x_2$ going clockwise, not including $x_1, x_2$, and let $A_1$ be the set of vertices from $x_2$ to $x_1$ going clockwise, not including $x_1, x_2$. Let $D$ be the set of edges between $A_0$ and $A_1$. By Proposition~\ref{prop:subgraph-e}, the graph $(V(C_{n}^{d_1,d_2, \ldots, d_k}),E(C_{n}^{d_1,d_2, \ldots, d_k})-D)$ is a square. 

Because $C_n^{d_1, d_2, \dotsc, d_k}$ is regular with degree $2k$ if $d_k \neq \frac{n}{2}$ and degree $2k-1$ if $d_k = \frac{n}{2}$, $C_n^{d_1, d_2, \dotsc, d_k}$ has $nk$ edges in the first case and $n(k-1)$ edges in the second case. We will now find the number of edges in $D$. Assume without loss of generality that $x_1 = 0$ and $x_2 = \frac{n}{2}$ and that $A_0 = \{1, 2, \dotsc, \frac{n}{2}-1\}$ and $A_1 = \{\frac{n}{2}+1, \dotsc, n-1\}$. Let $1 < d_\ell < \frac{n}{2}$ be given. The edge $ij$ exists, by definition, exactly when $i-j = \pm d_\ell$. This occurs exactly $2(d_\ell-1)$ times for some pair $(i,j)$ with $i \in A_0$ and $j \in A_1$. If $d_\ell = \frac{n}{2}$, this occurs exactly $d_\ell-1 = \frac{n}{2} - 1$ times.

In total, if $d_k \neq \frac{n}{2}$, the subgraph with $n$ vertices has $nk - 2 \sum_{i=1}^k (d_i-1)$ edges, and if $d_k = \frac{n}{2}$, the subgraph has $n(k-1) - \sum_{i=1}^{k-1} (d_i-1) - \left(\frac{n}{2}-1\right) = n\left(k-\frac{1}{2}\right)-2\sum_{i=1}^{k-1}(d_i-1)$ edges.
\end{proof}

Proposition~\ref{prop:circ-D} is illustrated in Figure~\ref{fig:circ-minus-e} for $C_8^{1,2}$. 

\begin{figure}[ht] 
\begin{center} 
\begin{tikzpicture}[scale=.7]
\draw[thick, black] (2.5,0)--(1,.8)--(.5,2)--(1,3.2)--(2.5,4)--(4,3.2)--(4.5,2)--(4,.8)--(2.5,0);
\draw[thick, black] (.5,2)--(2.5,4)--(4.5,2)--(2.5,0)--(.5,2);
\draw[thick, black] (1,.8)--(1,3.2);
\draw[thick, black] (4,.8)--(4,3.2);
\filldraw[red]
(2.5,4) circle [radius=6pt]
;
\filldraw[red]
(2.5,0) circle [radius=6pt]
;

\filldraw[black]
(.5,2) circle [radius=6pt]
(1,3.2) circle [radius=6pt]
(1,.8) circle [radius=6pt]
;
\filldraw[gray]
(4.5,2) circle [radius=6pt]
(4,3.2) circle [radius=6pt]
(4,.8) circle [radius=6pt]
;

\node(0) at (.75,3.8) {$v_1$};
\node(0) at (4.2,3.8) {$v'_1$};
\node(0) at (.75,.2) {$v_2$};
\node(0) at (4.2,.2) {$v'_2$};

\end{tikzpicture}

\caption{$C_8^{1,2}$ with the edges $D=\{v_1v'_1, v_2v'_2\}$ deleted is a square, as in Proposition~\ref{prop:circ-D}. }

\label{fig:circ-minus-e}

\end{center}
\end{figure}
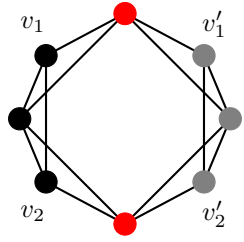 

\begin{theorem} \label{thm:circ-1-3}
Let $n\geq 7$. 
The circulant $C_{n}^{1,3}$ is a square if $n=8,10$ and is not a square otherwise. 
\end{theorem}
\begin{proof}

We first address the cases $n=7,9$. If $n=7$, the circulant $C_7^{1,3}$ is the complement of $C_7$, and therefore its automorphism group is the dihedral group $D_7$. The only involutions in $D_7$ are the reflections, and each has exactly one fixed point. Following the proof in Proposition~\ref{pr:cycles},  $C_7^{1,3}$ is not a square. 

If $n=9$, the circulant $C_9^{1,3}$ has exactly three triangles: $\{0,3,6\}$,
$\{1,4,7\}$, and $\{2,5,8\}$. Since $3$ and $6$ are adjacent, they cannot be twins. Assuming that $0$ is fixed and $1$ is unlabeled, then the set $\{0,3,6\}$ must be fixed under any involution, and therefore $3,6$ are fixed also. Hence the twin of $1$ is $8$. Since $N(1)\cap N(8) =\{0,2,7\}$, then $2,7$ are fixed and $4,5$ are unlabeled. However, $4,5$ are adjacent, and cannot be twins, and there are no other vertices that they can twin with. This contradiction shows $C_9^{1,3}$ is not a square. 

When $n$ is even, since $1,3$ are odd, $C_n^{1,3}$ is bipartite where the bipartition is the even vertices and the odd vertices. The circulant $C_8^{1,3}$ is $K_{4,4}$, which is a square by Proposition~\ref{pr:multipartite}. 
For the circulant $C_{10}^{1,3}$, let $\phi$ be involution that interchanges $1$ and $9$,  interchanges $4$ and $6$, and fixes the other 6 vertices. Since $N(1)=\{0,2,4,8\}$ and $N(9)=\{0,2,6,8\}$, then $\phi(N(1))=N(9)$ and vice versa. Similarly, $\phi(N(4))=\phi(\{1,3,5,7\})=N(6)=\{3,5,7,9\}$, and it is straightforward to check that $N(0),$ $N(2),$ $N(8)$ contain both $1$ and $9$, and 
$N(3), N(5), N(7)$ contain both $4$ and $6$, and that $\phi$ is an involution with the properties described in Observation~\ref{obs:alt-defn}. Thus, 
$C_{10}^{1,3}$ is a square.

Next, we show that if $n\geq 11$, then the circulant $C_n^{1,3}$ is not a square. Without loss of generality, let $0$ be a fixed vertex, and $1$ be an unlabeled vertex. Then $1$'s twin is in $N(0)=\{1,-1,3,-3\}$. 

First, assume that $1$'s twin is $-3$. Then either $-1$ and $3$ are twins, or both are fixed. We have $N(1)=\{-2,0,2,4\}$ and $N(-3)=\{-6,-4,-2,0\}$. Since $n\geq 11$, we have $N(1)\cap N(-3) =\{-2,0\}$, so $-2$ is fixed, and by Observation~\ref{obs:twin}, the vertices $-6,-4, 2,4$ are unlabeled. Since $-1$ is adjacent to $-2$, but $3$ is not, then $-1$ and $3$ are both fixed. 
Now $2$ is adjacent to $1,-1,3$, so $2$'s twin is adjacent to $-3,-1,3$, but $N(-3)\cap (N(-1)\cap N(3)) = \{-6,-4,-2,0\}\cap \{-2,0\}=\{-2,0\}$, hence $2$ does not have a twin, and this is a contradiction, so $-3$ is not $1$'s twin. 

We next assume that $1$'s twin is $-1$. Then either $-3$ and $3$ are twins, or both are fixed. We have $N(1)=\{-2,0,2,4\}$ and $N(-)=\{-4,-2,0,2\}$. 
Since $n\neq 8$, $N(1)\cap N(-1) =\{-2,0,2\}$, and $-2,2$ are both fixed, and $4,-4$ are unlabeled. Since $3$ is adjacent to $2$, but $-3$ is not, both $-3,3$ are fixed. Since $4$ is adjacent to $1,3$, its twin is adjacent to $-1,3$, but, since $n\geq 11$, we have $N(-1)\cap N(3) =\{-4,-2,0,2\}\cap \{0,2,4,6\}=\{0,2\}$, and hence $4$ does not have a twin. This contradiction shows that $1$ and $-1$ are not twins.

Finally, assume that $1$'s twin is $3$. (Note that this case is symmetric to $1$ being $-1$'s twin.) Then either $-1$ and $-3$ are twins, or both are fixed. We have $N(1)=\{-2,0,2,4\}$ and $N(3)=\{0,2,4,6\}$, and since $n\geq 8$, then $N(1)\cap N(3) =\{0,2,4\}$, so $2,4$ are fixed and $-2, 6$ are unlabeled. Since $-1$ is adjacent to $2$ and $-3$ is not, both $-1,-3$ are fixed. Since $-2$ is adjacent to $1, -1$, its twin 
is adjacent to $3,-1$, but since $n\geq 11$, we have $N(3)\cap N(-1)=\{0,2,4,6\}\cap \{-4,-2,0,2\} = \{0,2\} $ and hence $-2$ does not have a twin. This contradiction shows that $1$'s twin is not $3$.

Hence $C_n^{1,3}$ is not a square for $n\geq 11$, and we have shown that 
$C_n^{1,3}$ is a square for $n=8,10$ and not for any other integer greater than or equal to 7. \end{proof}

In the next theorem, we prove that the circulant $C_{n}^{1,d}$ is not a square graph if $n$ is large enough. 

\begin{theorem} \label{thm:circ-1-d} Let $d\geq 4$ and $n\geq d^2+1$.  The circulant $C_{n}^{1,d}$ is not a square graph. 
\end{theorem}

\begin{proof}  We do a proof by contradiction. Suppose that $\phi$ is an automorphism of 
$C_n^{1,d}$ that satisfies the conditions of Observation~\ref{obs:alt-defn}. Because $1$ and $n$ are relatively prime, there must be two vertices $i$ and $i+1$ such that $i$ is labeled and $i+1$ is not labeled.  Without loss of generality, let vertex $0\in F_{\phi}$ and $1\not\in F_{\phi}$. 
Since $\phi$ is an automorphism, $\phi(1)\in N(0)$, so $\phi(1)=-1, d$ or $-d$. We complete a proof to show that we reach a contradiction in each of these three cases, and thus there is no such automorphism $\phi$. The proof methods of the cases are similar. 

\noindent {\bf Case 1:} $\phi(1)=d$. Since $N(1)\cap N(d)=\{0,d+1\}$, then $d+1\in F_{\phi}$ by (3) of Observation~\ref{obs:alt-defn}. By induction, assume that $\phi(\ell)=\ell d$ and  $\phi(d+\ell)=\ell d+1$ for $1\leq \ell\leq j$. We will prove that
$\phi(j+1)=(j+1)d$, and $\phi(d+(j+1))=(j+1)d+1$ to complete the induction. 

Since $\phi$ is an automorphism, $\{j, j+1\}$ is an edge and $\phi(j)=jd$, we have $\phi(j+1)\in N(jd)=\{(j-1)d, \;jd-1,\; jd+1,\; (j+1)d\}$. By hypothesis, $\phi((j-1)d=j-1$ and $\phi(jd+1)=d+j$, so $\phi(j+1)$ is either $jd-1$ or $(j+1)d$. 
Again, since $\phi$ is an automorphism, $\{d+j, d+(j+1)\}$ is an edge and $\phi(d+j)=jd+1$, then $\phi(d+(j+1))\in N(jd+1)=\{(j-1)d+1,\; jd,\; jd+2, \;(j+1)d+1\}$. By hypothesis, $
\phi((j-1)d+1)=d+(j-1)$ and $\phi(jd)=j$, so $\phi(d+(j+1))$ is $jd+2$ or $(j+1)d+1$. Since $\{j+1, d+(j+1)\}$ is an edge, and $\phi$ is an automorphism, then $\{\phi(j+1),\phi(d+(j+1))\}$ is an edge. The only adjacent pair from $jd-1, (j+1)d$ and $jd+2, (j+1)d+1$ is $(j+1)d$ and $(j+1)d+1$, 
hence $\phi(j+1)=(j+1)d$ and $\phi(d+(j+1))=(j+1)d+1$. Thus, the induction is complete. 

We now consider what happens when $j=d-1$. Since $\{d-1,d\}$ is an edge and we know that $\phi(d-1)=(d-1)d$ and $\phi(d)=1$, it must be that $\{(d-1)d, 1\}$ is an edge. 
However, $N(d(d-1)) =\{d(d-2), \;d(d-1)-1, \;d(d-1)+1, \;d^2\}$. In each possible case, we get a contradiction.  
If $1=d(d-2)$, then $n=d^2-2d-1<d^2+1$; if $1=d(d-1)-1$ then $n=d^2-d-2<d^2+1$; if $1=d(d-1)+1$, then $n=d^2-d<d^2+1$; and if $1=d^2$, then $n=d^2-1<d^2+1$. 
Because of the contradictions, $\phi(1)$ is not $d$.  

\noindent {\bf Case 2:} $\phi(1)=-d$. Since $N(1)\cap N(-d)=\{0,-d+1\}$, then $-d+1\in F_{\phi}$ by (3) of Observation~\ref{obs:alt-defn}. It follows by an induction proof similar to that in Case~1 that $\phi(j)=-j d$, and $\phi(-d+j)=-j d+1$ (we put a minus sign in front of each occurrence of $d$). 

We consider what happens when $j=d-1$. We have $\{0,-1\}$ is an edge, and $\phi(0)=0$ and $\phi(-1)=-(d-1)d+1$, so it must be that $\{0, -(d-1)d+1\}$ is an edge. 
However, $N(-(d-1)d+1) =\{-d^2+1, -(d-1)d, -(d-1)d+2, -(d-2)d+1\}$, which implies  
$n=d^2-1, d^2-d, d^2-d-2$ or $d^2-2d-1$. Since these values are all less than $d^2+1$, this is a contradiction. 
Thus, $\phi(1)\neq -d$. 

\noindent {\bf Case 3:} $\phi(1)=-1$. The proof for this case is a bit different than the first two cases, since $1$ and $-1$ have only one common neighbor, but the methods are similar. Without loss of generality, let $1\in A_0$ and $-1\in A_{1}$. 
Since $N(1)\cap N(-1)=\{0\}$, then $N(1)-\{0\}=\{-d+1,\; 2,\; d+1\}\subseteq A_0$ and $N(-1)=\{-d-1,\;-2,\; d-1\}\subseteq A_1$. 
Thus, $d$ and $-d$ are each adjacent to vertices in both $A_0$ and $A_1$, and hence $-d,d\in F_{\phi}$. Since $d+1\in N(1)\cap N(d)$, then $\phi(d+1)\in N(-1)\cap N(d)=\{0,\; d-1\}$, and since $0\in F_{\phi}$, we have $\phi(d+1)=d-1$. 
By induction, assume that $\phi(\ell)=-\ell$ for $0\leq \ell \leq j$ and $\phi(d+\ell)=d-\ell$ for $1\leq \ell\leq j$, 
We will prove that $\phi(j+1)=-(j+1)$, and $\phi(d+(j+1))=d-(j+1)$ to complete the induction. 

Since $\phi$ is an automorphism, $\{j+1, d+(j+1)\}$ is an edge. Thus $\phi(j+1)\in N(\phi(j))=N(-j)=\{-j-d,\;-j-1,\;-(j-1),\;d-j\}$, and 
$\phi(d+(j+1))\in N(\phi(d+j))=N(d-j)=\{-j,\;d-(j+1),\;d-(j-1),\;2d-j\}$. By our inductive hypothesis, $\phi(j+1)\neq -(j-1)$ or $d-j $, so $\phi(j+1)$ is either $-j-d$ or $-j-1$, and also $\phi(d+(j+1))\neq -j $ or $d-(j-1)$, so
$\phi(d+(j+1))$ is either $d-(j+1)$ or $2d-j$. 
Since $\{j+1,\; d+(j+1)\}$ is an edge, and $\phi$ is an automorphism, then $\{\phi(j+1),\phi(d+(j+1))\}$ is an edge. 
The only adjacent pair from $-j-d,\; -(j+1)$ and $d-(j+1), \;2d-j$ is 
$-(j+1)$ and $d-(j+1)$, hence $\phi(j+1)=-(j+1)$ and $\phi(d+(j+1))=d-(j+1)$ and the induction is complete. 

Let $j=d-2$. Then $\phi(2d-2)=\phi(d+(d-2))=d-(d-2)=2$, but $\phi(2) =-2$, hence $-2=2d-2$, which implies that $n=2d<d^2+1$, a contradiction. Thus, $\phi(1)\neq -1$. Hence, no choice of involution works, and $C_{n}^{1,d}$
is not a square graph. 
\end{proof}

\subsection{Johnson Graphs} In this subsection, we consider the family of Johnson graphs. We begin with a formal definition.

\begin{definition}
The Johnson graph, $J(n,k)$ has vertex set the $k$-element subsets of $\{1,2,3, \ldots, n\}$ where there is an edge between two $k$-subsets if their intersection contains $k-1$ elements.  \cite{HS93}
\end{definition}

Note that $J(n,k)$ is isomorphic to $J(n,n-k)$, so generally we only consider $k$ in the range of $1 \leq k \leq \frac{n}{2}$.

\begin{example}
    We give two examples of Johnson graphs in Figure~\ref{fig:johnson-graphs}, $J(4,2)$ on the left and $J(5,3)$ on the right.
\end{example}

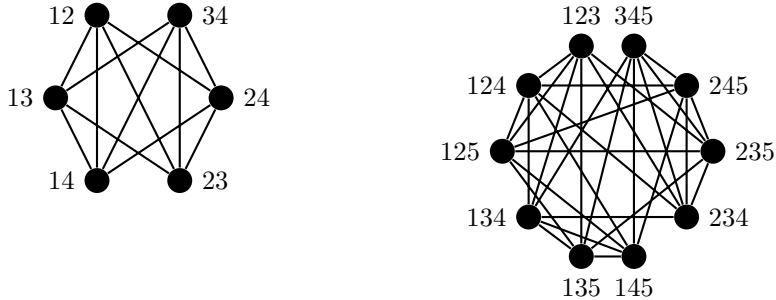
\begin{figure}[ht] 
\begin{center} 
\begin{multicols}{2}
\begin{tikzpicture}[scale=.55]
\node[circle,fill=black,label=left:{$12$}] (12) at (0,1) {};
\node[circle,fill=black,label=left:{$13$}] (13) at (-1,-1) {};
\node[circle,fill=black,label=left:{$14$}] (14) at (0,-3) {};
\node[circle,fill=black,label=right:{$23$}] (23) at (2,-3) {};
\node[circle,fill=black,label=right:{$24$}] (24) at (3,-1) {};
\node[circle,fill=black,label=right:{$34$}] (34) at (2,1) {};

\draw[black,thick] (24) -- (14) -- (12) -- (24) -- (23) -- (13) -- (34) -- (14) -- (13) -- (12) -- (23) -- (34) -- (24);
\end{tikzpicture}

\columnbreak
\begin{tikzpicture}[scale=.7]
\node[circle,fill=black,label=above:{$123$}] (123) at (0,0) {};
\node[circle,fill=black,label=left:{$124$}] (124) at (-1,-0.75) {};
\node[circle,fill=black,label=left:{$125$}] (125) at (-1.5,-2) {};
\node[circle,fill=black,label=left:{$134$}] (134) at (-1,-3.25) {};
\node[circle,fill=black,label=below:{$135$}] (135) at (0,-4) {};
\node[circle,fill=black,label=below:{$145$}] (145) at (1,-4) {};
\node[circle,fill=black,label=right:{$234$}] (234) at (2,-3.25) {};
\node[circle,fill=black,label=right:{$235$}] (235) at (2.5,-2) {};
\node[circle,fill=black,label=right:{$245$}] (245) at (2,-0.75) {};
\node[circle,fill=black,label=above:{$345$}] (345) at (1,0) {};

\draw[black,thick] (123) -- (124) -- (125) -- (123) -- (134);
\draw[black,thick] (145) -- (125) -- (135) -- (145) -- (345) -- (245) -- (235) -- (234) -- (134) -- (345) -- (135) -- (134) -- (124) -- (245) -- (234) -- (123) -- (235) -- (135) -- (123);
\draw[black,thick] (145) -- (245) -- (125) -- (235) -- (345) -- (234) -- (124) -- (145) -- (134);
\end{tikzpicture}
\end{multicols}

\caption{On the left is the graph $J(4,2)$. On the right is the graph $J(5,3)$.}
\label{fig:johnson-graphs}
\end{center}
\end{figure}

\begin{theorem}
The Johnson graph $J(n,k)$ is not a square when $n\neq 2k$.
\end{theorem}

\begin{proof}
    It is known from \cite{BCN89} that the automorphism group of $J(n,k)$ when $n \neq 2k$ is $S_n$, acting on the vertices of the graph by permuting the elements of the corresponding subsets. If $J(n,k)$ were a square, then there would be an involution as described in Observation~\ref{obs:alt-defn}. In $S_n$, the only possible involutions are the product of disjoint transpositions. Consider one such involution, say $(i_1,i_2)(i_3,i_4)\cdots(i_{2\ell-1},i_{2\ell})$ where the $i_j$'s are all distinct. Then vertices that contain exactly zero or both elements in each of the transpositions $(i_{2j-1},i_{2j})$ are fixed. Consider a vertex $v$ corresponding to the set $S$ containing exactly one of $(i_{2j-1},i_{2j})$ for some $j$, say $i_{2j-1}$. Then the vertex $v'$ corresponding to $S':=S\backslash \{i_{2j-1}\} \cup \{i_{2j}\}$ is adjacent to $v$ since their intersection has size $k-1$, and so our vertices cannot be partitioned in the way of Observation~\ref{obs:alt-defn}.
\end{proof}

\section{Building Square Graphs} \label{sec:build-graphs}

In this subsection, we consider graphs obtained by starting with some square graph and taking various graph operations: graph products and graph joins.

We begin this subsection by considering hypercubes, which are a specific case of a graph product: hypercubes can be obtained by taking a finite Cartesian product of $K_2$'s. We begin with this specific example as the proof illustrates the more general result that follows later. 

Recall that the $n$-hypercube, denoted $Q_n$, is the graph whose vertices are all binary $n$-tuples, where two vertices are adjacent if and only if they differ in exactly one coordinate.

\begin{theorem} \label{thm:hypercube}
Let $n\geq 2$. Then the hypercube $Q_n$ is a square. 
\end{theorem}

\begin{proof}
Let $\phi:V(Q_n)\rightarrow V(Q_n)$ where the first two coordinates of each vertex are 
interchanged, that is, $\phi(a_1, a_2, a_3,\ldots, a_n) = (a_2, a_1,a_3, \ldots, a_n)$. Then $\phi$ is an involution, and it is straightforward to show that $\phi$ is an automorphism of $Q_n$. We have $F_{\phi}
=\{(a_1, a_2, a_3,\ldots, a_n)\;|\;a_1=a_2\}$, and hence $F_{\phi}$ and $V(F_{\phi})$ are not empty. Let $A_0=\{(a_1, a_2, a_3,\ldots, a_n)\;|\;a_1=0, a_2=1\}$ and $A_1=\{(a_1, a_2, a_3,\ldots, a_n)\;|\;a_1=1, a_2=0\}$. Then $\phi(A_0)=A_1$ and there are no edges with one end in $A_0$ and one in $A_1$ because the first two digits of the vertices in $A_0$ are both different from the first two digits of each vertex in $A_1$. Since the first two digits are different, $N((0, 1, a_3,\ldots, a_n))\cap N((1, 0, a_3,\ldots, a_n))=\{(0,0,a_3, \ldots, a_n),(1,1,a_3, \ldots, a_n)\}\subseteq F_{\phi}$.  
Hence all conditions of Observation~\ref{obs:alt-defn} are met and $Q_n$ is a square. 
\end{proof}

In fact, products of graphs preserve squareness.  Below we consider Cartesian, direct, strong, and lexicographic products (see Figure \ref{fig:products} for illustrations of each of these products).  For each, we prove that the product of a square graph with any arbitrary graph is a square.  Additionally, in each case we show this condition is tight: there are pairs of non-square graphs whose products are \emph{not} squares.

We first generalize Theorem~\ref{thm:hypercube} to the Cartesian product of graphs. Recall that for two graphs $G$ and $H$, their Cartesian product, denoted $G \Box H$, is the graph with vertex set $V(G) \times V(H)$ and edge set such that $(g,h)$ is adjacent to $(g',h')$ exactly when either $g=g'$ and $hh' \in E(H)$ or $h=h'$ and $gg' \in E(G)$.

\begin{theorem} \label{thm:cartesian}
 Let $G$ be a square and $H$ be a graph. Then $G\Box H$ is a square.    
\end{theorem}

\begin{proof} Let $\phi:V(G)\rightarrow V(G)$ be an automorphism of $G$ that satisfies Observation~\ref{obs:alt-defn} with sets $F_{\phi}, A_0, A_1$. Define $\psi(V(G\Box H))$ where $\psi(g,h) = (\phi(g), h)$. Then $\psi$ is an involution since $\phi$ is an involution, and $F_{\psi}=\{(g,h)\;|\;g\in F_{\phi}\}$ and is not empty, since $F_{\phi}$ is not empty, and also $V(G\Box H)-F_{\psi}$ is not 
empty, since $V(G)-F_{\phi}$ is not empty. Let $B_0=\{(g,h)\;|\;g\in A_0\}$ and $B_1=\{(g,h)\;|\;g\in A_1\}$. It is straightforward to check that $\psi(B_0)=B_1$, there are no edges between the vertices in $B_0$ and $B_1$, and $V(G\Box H)-F_{\psi}=B_0\cup B_1$.

If $\{(g_1,h_1), (g_2, h_2)\}$ is an edge in $G\Box H$, then either $g_1=g_2$ and $\{h_1, h_2\}\in E(H)$, or $\{g_1, g_2\}\in E(G)$ and $h_1=h_2$. If $g_1=g_2$, then $\{\psi(g_1, h_1), \psi(g_1, h_2)\}=\{(\phi(g_1), h_1), (\phi(g_1), h_2)\}\in E(G\Box H)$, 
since $\{h_1,h_2\}\in E(H)$, and if $h_1=h_2$, then $\{\psi(g_1, h_1), \psi(g_2, h_1)\}=\{(\phi(g_1), h_1), (\phi(g_2),h_1)\}\in E(G\Box H)$, since we have $\{\phi(g_1), \phi(g_2)\}\in E(G)$. Thus, $\psi$ is an automorphism. 
If $(g,h)\in V(G\Box H)-F_{\psi}$, then $g\in V(G)-F_{\phi}$, and 
$N_{G\Box H}(g,h)\cap N_{G\Box H}(\phi(g),h)= \{(w,h)\;|\;w\in N_G(g)\cap N_G(\phi(g))\}$, and $w\in F_{\phi}$ by definition of $\phi$. Therefore, 
$N_{G\Box H}(g,h)\cap  N_{G\Box H}(\phi(g),h))\subseteq F_{\psi}$. 
Hence all conditions of Observation~\ref{obs:alt-defn} are met and $G\Box H$ is a square. 
\end{proof}

The next example shows that the hypothesis that $G$ is a square cannot, in general, be dropped. 

\begin{example} \label{exa:cartesian} \rm 
$K_3\Box K_3$ is not a square. Let the vertices of $K_3$ be $a,b,c$. Then $N(a,a)=\{(a,b), (a,c), (b,a), (c,a)\}$ and the remaining vertices $(b,b), (b,c),$ $ (c,b),$  $(c,c)$ are distance 2 from $(a,a)$. For a proof by contradiction, suppose that that $\phi$ is an automorphism that satisfies Observation~\ref{obs:alt-defn}. Without loss of generality, suppose that $\phi(a,a)=(b,b)$ and $(a,a)\in A_0$ and $(b,b)\in A_1$.  We have $N(b,b) =\{(b,a), (b,c), (a,b), (c,b)\}$, and $N(a,a)\cap N(b,b) = \{(a,b),(b,a) \}\subseteq F_{\phi}$, and $(a,c), (c,a), (b,c), (c,b) \not\in F_{\phi}$. However, this is a contradiction because $(a,a),(a,c),(b,c),(b,b)$ is a path in $K_3\Box K_3-F_{\phi}$ from $(a,a)\in A_0$ to $(b,b)\in A_1$. Hence $K_3\Box K_3$ is not a square. 
\end{example}

We note that it is possible for the Cartesian product of two non-square graphs to be a square. See the following example.

\begin{example} \rm 
    $K_2 \Box K_2 = C_4$ is a square by Proposition~\ref{pr:cycles}, but $K_2$ is not a square by Proposition~\ref{pr:complete}.
\end{example}

In fact, using the fact that $Q_n$ is the Cartesian product of $n$ copies of $P_2 = K_2$, an alternate proof that $Q_n$ is a square for $n\geq 2$ is to begin with the fact that $Q_2 = C_4 = K_2 \Box K_2$ and then use induction and Theorem~\ref{thm:cartesian}. 

We move on to a second type of graph product: the tensor product (also referred to as the direct product). Recall that the tensor product of two graphs $G$ and $H$, denoted $G \times H$, is the graph with vertex set $V(G) \times V(H)$ and edge set such that $(g,h)$ is adjacent to $(g',h')$ if and only if both $gg' \in E(G)$ and $hh' \in E(H)$. 

\begin{theorem} \label{thm:tensor}
 Let $G$ be a square and $H$ be a graph. Then $G\times H$ is a square.    
\end{theorem}

\begin{proof}
    Let $\phi:V(G)\rightarrow V(G)$ be an automorphism of $G$ that satisfies Observation~\ref{obs:alt-defn} with sets $F_{\phi}, A_0, A_1$. Define $\psi(V(G \times H))$ where $\psi(g,h) = (\phi(g), h)$. Define $F_{\psi}$,  $B_0$, and $B_1$ as in the proof of Theorem~\ref{thm:cartesian}. Following the same argument, $\psi$ is an involution, $F_{\psi} \ne \emptyset$, $\psi(B_0)=B_1$, there are no edges between the vertices in $B_0$ and $B_1$, and $V(G\times H)-F_{\psi}=B_0\cup B_1$.

    If $\{(g_1,h_1), (g_2,h_2)\} \in E(G \times H)$, then $(g_1,g_2) \in E(G)$ and $(h_1,h_2) \in E(H)$. Then $\{\psi(g_1,h_1), \psi(g_2,h_2)\} = \{(\phi(g_1),h_1), (\phi(g_2),h_2)\} \in E(G \times H)$ since $(g_1,g_2) \in E(G)$ implies $(\phi(g_1), \phi(g_2)) \in E(G)$,   and $(h_1,h_2) \in E(H)$. Thus, $\psi$ is an automorphism. If $(g,h) \in V(G \times H) - F_{\psi}$ then $g \in V(G) - F_{\phi}$, and $N_{G\times H}(g,h) \cap N_{G\times H}(\phi(g),h) = \{(w,h') | w \in N_G(g) \cap N_G (\phi(g), h' \in N_H(h) \}$. By definition of $\phi$, $w \in F_{\phi}$, so  $N_{G\times H}(g,h) \cap N_{G\times H}(\phi(g),h) \subseteq F_{\psi}$. Hence all conditions of Observation~\ref{obs:alt-defn} are met and $G\times H$ is a square. 
\end{proof}

We remark that if neither $G$ nor $H$ is a square graph, then $G\times H$ need not be a square graph. The tensor product $K_2\times K_2$ is two disjoint edges. By Observation~\ref{obs:alt-defn}, the involution that interchanges the two component $K_2$'s is not a butterfly involution, since $F_{\phi}$ is empty. Hence it is not a square graph.

We now consider a third type of graph product: the strong product.
Recall that the strong product of two graphs $G$ and $H$, denoted $G \boxtimes H$, is the graph with vertex set $V(G) \times V(H)$ such that $(g,h)$ is adjacent to $(g',h')$ exactly when one of the following is true: $g=g'$ and $hh'\in E(H)$, $h=h'$ and $gg' \in E(G)$, or $gg' \in E(G)$ and $hh' \in E(H)$.

\begin{theorem} \label{thm:strong}
 Let $G$ be a square and $H$ be a graph. Then $G\boxtimes H$ is a square.    
\end{theorem}

\begin{proof}
Let $\phi:V(G)\rightarrow V(G)$ be an automorphism of $G$ that satisfies Observation~\ref{obs:alt-defn} with sets $F_{\phi}, A_0, A_1$. Define $\psi(V(G \boxtimes H))$ where $\psi(g,h) = (\phi(g), h)$. Define $F_{\psi}$,  $B_0$, and $B_1$ as in the proof of Theorem~\ref{thm:cartesian}. Following the same argument, $\psi$ is an involution, $F_{\psi} \ne \emptyset$, $\psi(B_0)=B_1$, there are no edges between the vertices in $B_0$ and $B_1$, and $V(G\boxtimes H)-F_{\psi}=B_0\cup B_1$.

If $\{(g_1,h_1), (g_2,h_2)\} \in E(G \boxtimes H)$, then either $g_1 = g_2$ and $(h_1,h_2) \in E(H)$, or $h_1 = h_2$ and $(g_1,g_2) \in E(G)$, or $(g_1,g_2) \in E(G)$ and $(h_1,h_2) \in E(H)$. If $g_1 = g_2$ then $\{\psi(g_1,h_1), \psi(g_2,h_2)\} = \{(\phi(g_1),h_1), (\phi(g_1),h_2)\} \in E(G \boxtimes H)$ since $(h_1,h_2) \in E(H)$.  If $h_1 = h_2$ then since $(g_1,g_2) \in E(G)$ implies $(\phi(g_1),\phi(g_2)) \in E(G)$, we have $\{\psi(g_1,h_1), \psi(g_2,h_2)\} = \{(\phi(g_1),h_1), (\phi(g_2),h_1)\} \in E(G \boxtimes H)$. 
Finally, if 
$(g_1,g_2) \in E(G)$ and $(h_1,h_2) \in E(H)$, then since $(g_1,g_2) \in E(G)$ implies $(\phi(g_1),\phi(g_2)) \in E(G)$, we have $\{\psi(g_1,h_1), \psi(g_2,h_2)\}=\{(\phi(g_1),h_1),$  $ (\phi(g_2),h_2)\} \in E(G \boxtimes H)$. Thus, $\psi$ is an automorphism. If $(g,h) \in V(G \boxtimes H) - F_{\psi}$ then $g \in V(G) - F_{\phi}$, and $N_{G\boxtimes H}(g,h) \cap N_{G\boxtimes H}(\phi(g),h) = \{(w,h') | w \in N_G(g) \cap N_G (\phi(g), h' \in N_H(h)\cup\{h\} \}$. By definition of $\phi$, $w \in F_{\phi}$, so  $N_{G\boxtimes H}(g,h) \cap N_{G\boxtimes H}(\phi(g),h) \subseteq F_{\psi}$. Hence all conditions of Observation~\ref{obs:alt-defn} are met and $G\boxtimes H$ is a square. 
\end{proof}

The previous theorem is tight. If the condition that one factor in the strong product is a square is dropped, the result need not be a square graph. For example, we have that $K_2$ is not a square, and $K_2\boxtimes K_2$ is isomorphic to $K_4$, which is not a square.

Finally, we consider the fourth standard graph product, the lexicographic product. This product is not commutative, unlike the previous ones. Recall that for two graphs $G$ and $H$, their lexicographic product, denoted $G \circ H$, is the graph with vertex set $V(G) \times V(H)$ and edge set such that $(g,h)$ is adjacent to $(g',h')$ exactly when either $gg' \in E(G)$, or $g=g'$ and $hh' \in E(H)$.

\begin{theorem} \label{thm:lexicographic}
 Let $G$ be a square and $H$ be a graph. Then both $G\circ H$ is a square and $H\circ G$ is a square.
\end{theorem}

\begin{proof}
We prove the first half of the statement. Let $\phi:V(G)\rightarrow V(G)$ be a butterfly involution of $G$ that satisfies Observation~\ref{obs:alt-defn} with sets $F_{\phi}, A_0, A_1$. Define $\psi(V(G \circ H))$ where $\psi(g,h) = (\phi(g), h)$. Define $F_{\psi}$,  $B_0$, and $B_1$ as in the proof of Theorem~\ref{thm:cartesian}. Following the same argument, $\psi$ is an involution, $F_{\psi} \ne \emptyset$, $\psi(B_0)=B_1$, there are no edges between the vertices in $B_0$ and $B_1$, and $V(G\circ H)-F_{\psi}=B_0\cup B_1$.

If $\{(g_1,h_1), (g_2,h_2)\} \in E(G \circ H)$, then either  $(g_1,g_2) \in E(G)$, or $g_1 = g_2$ and $(h_1,h_2) \in E(H)$. If $g_1 = g_2$ then since $(h_1,h_2) \in E(H)$, $\{\psi(g_1,h_1), \psi(g_2,h_2)\} = \{(\phi(g_1),h_1), (\phi(g_1),h_2)\} \in E(G \circ H)$.  If  $(g_1,g_2) \in E(G)$, then  $(\phi(g_1),\phi(g_2)) \in E(G)$, and $\{\psi(g_1,h_1), \psi(g_2,h_2)\} = \{(\phi(g_1),h_1), (\phi(g_2),h_2)\} \in E(G \circ H)$.   Thus, $\psi$ is an automorphism. If $(g,h) \in V(G \circ H) - F_{\psi}$ then $g \in V(G) - F_{\phi}$, and $N_{G\circ H}(g,h) \cap N_{G\circ H}(\phi(g),h) = \{(w,h') | w \in N_G(g) \cap N_G (\phi(g), h' \in N_H(h)\cup\{h\} \}$. By definition of $\phi$, $w \in F_{\phi}$, so  $N_{G\circ H}(g,h) \cap N_{G\circ H}(\phi(g),h) \subseteq F_{\psi}$. Hence all conditions of Observation~\ref{obs:alt-defn} are met and $G\circ H$ is a square. 

For the second half, let $u\in V(H)$ and define $\sigma$ acting on $ V(H \circ G)$ where $\sigma(h,g) = (h,g)$ for $h\neq u$ and $\sigma(u,g) = (u,\phi(g))$. The vertices with first coordinate $u$ induce a subgraph $\hat{G}$ isomorphic to $G$, and by the definition of the lexicographic product, every vertex in $\hat{G}$ has the same set of neighbors in $H\circ G-\hat{G}$. Thus, $\sigma$ is an involution, and since $\phi$ is a butterfly involution of $G$, $\sigma$ is a butterfly involution of $H\circ G$. 
\end{proof}

The previous theorem is tight. If the condition that one factor in the lexicographic product is a square is dropped, the result need not be a square graph. For example, we have that $K_2$ is not a square, and $K_2\circ K_2$ is isomorphic to $K_4$, which is not a square.

We provide an illustration of our results for the four graph products in the next example, illustrated in Figure~\ref{fig:GH} and Figure~\ref{fig:products}. 

\begin{example} \rm 
Let $G = C_4$ and $H=P_3$, as illustrated in Figure~\ref{fig:GH}, and let $\phi$ be the involution of $C_4$ such that $\phi(g_1)=g_1$, $\phi(g_2) = g_4$, $\phi(g_3)=g_3$, and $\phi(g_4)=g_2$. We color code these such that the fixed vertices in $C_4$ are colored yellow, $g_2$ is colored red, and $g_4$ is colored blue. The four graph products of $C_4$ and $P_3$ are illustrated in Figure~\ref{fig:products}. The copies of $g_i$ are colored the same as in $C_4$ to help visualize what happens to both the fixed sets and the vertices in $A_0$ and $A_1$ under the graph product operation.

\begin{figure}[ht] 
    \centering
    \begin{multicols}{2}
    \begin{tikzpicture}
    \node[circle,draw=black,fill=yellow,label=above:{$g_1$},scale=0.8] (g1) at (0,0) {};
    \node[circle,draw=black,fill=red,label=above:{$g_2$},scale=0.8] (g2) at (1,0) {};
    \node[circle,draw=black,fill=yellow,label=below:{$g_3$},scale=0.8] (g3) at (1,-1) {};
    \node[circle,draw=black,fill=blue,label=below:{$g_4$},scale=0.8] (g4) at (0,-1) {};

    \node[label=below:{$G=C_4$}] at (0.5,-1.75) {};

    \draw[black,thick] (g1) -- (g2) -- (g3) -- (g4) -- (g1);
\end{tikzpicture}
\columnbreak

\begin{tikzpicture}
    \node[circle,draw=black,fill=black,label=above:{$h_1$},scale=0.8] (h1) at (0,0) {};
    \node[circle,draw=black,fill=black,label=above:{$h_2$},scale=0.8] (h2) at (1,0) {};
    \node[circle,draw=black,fill=black,label=above:{$h_3$},scale=0.8] (h3) at (2,0) {};

    \node[label=below:{$H=P_3$}] at (1,-0.75) {};
    \node at (1,1) {};

    \draw[black,thick] (h1) -- (h2) -- (h3);
\end{tikzpicture}
    \end{multicols}
    \caption{$G=C_4$ and $H=P_3$. $G$ is colored so that the fixed points under the involution $\phi$ are yellow.} \label{fig:GH}
\end{figure}
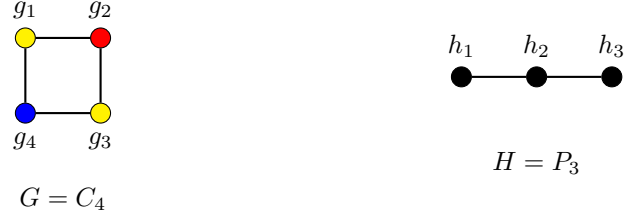

\begin{figure}[ht] 
\centering
\begin{multicols}{4}
\begin{tikzpicture}[scale=0.7] 
    \node[circle,draw=black,fill=yellow,scale=0.8] (g1h1) at (0,0) {};
    \node[circle,draw=black,fill=yellow,scale=0.8] (g1h2) at (1,0) {};
    \node[circle,draw=black,fill=yellow,scale=0.8] (g1h3) at (2,0) {};
    \node[circle,draw=black,fill=red,scale=0.8] (g2h1) at (0,-1) {};
    \node[circle,draw=black,fill=red,scale=0.8] (g2h2) at (1,-1) {};
    \node[circle,draw=black,fill=red,scale=0.8] (g2h3) at (2,-1) {};
    \node[circle,draw=black,fill=yellow,scale=0.8] (g3h1) at (0,-2) {};
    \node[circle,draw=black,fill=yellow,scale=0.8] (g3h2) at (1,-2) {};
    \node[circle,draw=black,fill=yellow,scale=0.8] (g3h3) at (2,-2) {};
    \node[circle,draw=black,fill=blue,scale=0.8] (g4h1) at (0,-3) {};
    \node[circle,draw=black,fill=blue,scale=0.8] (g4h2) at (1,-3) {};
    \node[circle,draw=black,fill=blue,scale=0.8] (g4h3) at (2,-3) {};

    \node[label=above:{$h_1$}] at (0,0.1) {};
    \node[label=above:{$h_2$}] at (1,0.1) {};
    \node[label=above:{$h_3$}] at (2,0.1) {};
    \node[label=left:{$g_1$}] at (-0.1,0) {};
    \node[label=left:{$g_2$}] at (-0.1,-1) {};
    \node[label=left:{$g_3$}] at (-0.1,-2) {};
    \node[label=left:{$g_4$}] at (-0.1,-3) {};

    \draw[black,thick] (g1h1) -- (g1h2) -- (g1h3);
    \draw[black,thick] (g2h1) -- (g2h2) -- (g2h3);
    \draw[black,thick] (g3h1) -- (g3h2) -- (g3h3);
    \draw[black,thick] (g4h1) -- (g4h2) -- (g4h3);
    \draw[black,thick] (g1h1) -- (g2h1) -- (g3h1) -- (g4h1);
    \draw[black,thick] (g1h2) -- (g2h2) -- (g3h2) -- (g4h2);
    \draw[black,thick] (g1h3) -- (g2h3) -- (g3h3) -- (g4h3);
    \draw[black,thick] (g1h1) to [out=-105,in=105] (g4h1);
    \draw[black,thick] (g1h2) to [out=-105,in=105] (g4h2);
    \draw[black,thick] (g1h3) to [out=-75,in=75] (g4h3);

    \node[label=below:{$C_4 \square P_3$}] at (1,-3.5) {};
\end{tikzpicture}

\begin{tikzpicture}[scale=0.7] 
    \node[circle,draw=black,fill=yellow,scale=0.8] (g1h1) at (0,0) {};
    \node[circle,draw=black,fill=yellow,scale=0.8] (g1h2) at (1,0) {};
    \node[circle,draw=black,fill=yellow,scale=0.8] (g1h3) at (2,0) {};
    \node[circle,draw=black,fill=red,scale=0.8] (g2h1) at (0,-1) {};
    \node[circle,draw=black,fill=red,scale=0.8] (g2h2) at (1,-1) {};
    \node[circle,draw=black,fill=red,scale=0.8] (g2h3) at (2,-1) {};
    \node[circle,draw=black,fill=yellow,scale=0.8] (g3h1) at (0,-2) {};
    \node[circle,draw=black,fill=yellow,scale=0.8] (g3h2) at (1,-2) {};
    \node[circle,draw=black,fill=yellow,scale=0.8] (g3h3) at (2,-2) {};
    \node[circle,draw=black,fill=blue,scale=0.8] (g4h1) at (0,-3) {};
    \node[circle,draw=black,fill=blue,scale=0.8] (g4h2) at (1,-3) {};
    \node[circle,draw=black,fill=blue,scale=0.8] (g4h3) at (2,-3) {};

    \node[label=above:{$h_1$}] at (0,0.1) {};
    \node[label=above:{$h_2$}] at (1,0.1) {};
    \node[label=above:{$h_3$}] at (2,0.1) {};
    \node[label=left:{$g_1$}] at (-0.1,0) {};
    \node[label=left:{$g_2$}] at (-0.1,-1) {};
    \node[label=left:{$g_3$}] at (-0.1,-2) {};
    \node[label=left:{$g_4$}] at (-0.1,-3) {};

    \draw[black,thick] (g1h1) -- (g2h2) -- (g3h1) -- (g4h2) -- (g1h1);
    \draw[black,thick] (g1h2) -- (g2h1) -- (g3h2) -- (g4h1) -- (g1h2);
    \draw[black,thick] (g1h2) -- (g2h3) -- (g3h2) -- (g4h3) -- (g1h2);
    \draw[black,thick] (g1h3) -- (g2h2) -- (g3h3) -- (g4h2) -- (g1h3);

    \node[label=below:{$C_4 \times P_3$}] at (1,-3.5) {};
\end{tikzpicture}

\begin{tikzpicture}[scale=0.7] 
    \node[circle,draw=black,fill=yellow,scale=0.8] (g1h1) at (0,0) {};
    \node[circle,draw=black,fill=yellow,scale=0.8] (g1h2) at (1,0) {};
    \node[circle,draw=black,fill=yellow,scale=0.8] (g1h3) at (2,0) {};
    \node[circle,draw=black,fill=red,scale=0.8] (g2h1) at (0,-1) {};
    \node[circle,draw=black,fill=red,scale=0.8] (g2h2) at (1,-1) {};
    \node[circle,draw=black,fill=red,scale=0.8] (g2h3) at (2,-1) {};
    \node[circle,draw=black,fill=yellow,scale=0.8] (g3h1) at (0,-2) {};
    \node[circle,draw=black,fill=yellow,scale=0.8] (g3h2) at (1,-2) {};
    \node[circle,draw=black,fill=yellow,scale=0.8] (g3h3) at (2,-2) {};
    \node[circle,draw=black,fill=blue,scale=0.8] (g4h1) at (0,-3) {};
    \node[circle,draw=black,fill=blue,scale=0.8] (g4h2) at (1,-3) {};
    \node[circle,draw=black,fill=blue,scale=0.8] (g4h3) at (2,-3) {};

    \node[label=above:{$h_1$}] at (0,0.1) {};
    \node[label=above:{$h_2$}] at (1,0.1) {};
    \node[label=above:{$h_3$}] at (2,0.1) {};
    \node[label=left:{$g_1$}] at (-0.1,0) {};
    \node[label=left:{$g_2$}] at (-0.1,-1) {};
    \node[label=left:{$g_3$}] at (-0.1,-2) {};
    \node[label=left:{$g_4$}] at (-0.1,-3) {};

    \draw[black,thick] (g1h1) -- (g1h2) -- (g1h3);
    \draw[black,thick] (g2h1) -- (g2h2) -- (g2h3);
    \draw[black,thick] (g3h1) -- (g3h2) -- (g3h3);
    \draw[black,thick] (g4h1) -- (g4h2) -- (g4h3);
    \draw[black,thick] (g1h1) -- (g2h1) -- (g3h1) -- (g4h1);
    \draw[black,thick] (g1h2) -- (g2h2) -- (g3h2) -- (g4h2);
    \draw[black,thick] (g1h3) -- (g2h3) -- (g3h3) -- (g4h3);
    \draw[black,thick] (g1h1) to [out=-105,in=105] (g4h1);
    \draw[black,thick] (g1h2) to [out=-105,in=105] (g4h2);
    \draw[black,thick] (g1h3) to [out=-75,in=75] (g4h3);
    \draw[black,thick] (g1h1) -- (g2h2) -- (g3h1) -- (g4h2) -- (g1h1);
    \draw[black,thick] (g1h2) -- (g2h1) -- (g3h2) -- (g4h1) -- (g1h2);
    \draw[black,thick] (g1h2) -- (g2h3) -- (g3h2) -- (g4h3) -- (g1h2);
    \draw[black,thick] (g1h3) -- (g2h2) -- (g3h3) -- (g4h2) -- (g1h3);

    \node[label=below:{$C_4 \boxtimes P_3$}] at (1,-3.5) {};
\end{tikzpicture}

\begin{tikzpicture}[scale=0.7] 
    \node[circle,draw=black,fill=yellow,scale=0.8] (g1h1) at (0,0) {};
    \node[circle,draw=black,fill=yellow,scale=0.8] (g1h2) at (1,0) {};
    \node[circle,draw=black,fill=yellow,scale=0.8] (g1h3) at (2,0) {};
    \node[circle,draw=black,fill=red,scale=0.8] (g2h1) at (0,-1) {};
    \node[circle,draw=black,fill=red,scale=0.8] (g2h2) at (1,-1) {};
    \node[circle,draw=black,fill=red,scale=0.8] (g2h3) at (2,-1) {};
    \node[circle,draw=black,fill=yellow,scale=0.8] (g3h1) at (0,-2) {};
    \node[circle,draw=black,fill=yellow,scale=0.8] (g3h2) at (1,-2) {};
    \node[circle,draw=black,fill=yellow,scale=0.8] (g3h3) at (2,-2) {};
    \node[circle,draw=black,fill=blue,scale=0.8] (g4h1) at (0,-3) {};
    \node[circle,draw=black,fill=blue,scale=0.8] (g4h2) at (1,-3) {};
    \node[circle,draw=black,fill=blue,scale=0.8] (g4h3) at (2,-3) {};

    \node[label=above:{$h_1$}] at (0,0.1) {};
    \node[label=above:{$h_2$}] at (1,0.1) {};
    \node[label=above:{$h_3$}] at (2,0.1) {};
    \node[label=left:{$g_1$}] at (-0.1,0) {};
    \node[label=left:{$g_2$}] at (-0.1,-1) {};
    \node[label=left:{$g_3$}] at (-0.1,-2) {};
    \node[label=left:{$g_4$}] at (-0.1,-3) {};

    \draw[black,thick] (g1h1) -- (g1h2) -- (g1h3);
    \draw[black,thick] (g2h1) -- (g2h2) -- (g2h3);
    \draw[black,thick] (g3h1) -- (g3h2) -- (g3h3);
    \draw[black,thick] (g4h1) -- (g4h2) -- (g4h3);
    \draw[black,thick] (g1h1) -- (g2h1) -- (g3h1) -- (g4h1);
    \draw[black,thick] (g1h2) -- (g2h2) -- (g3h2) -- (g4h2);
    \draw[black,thick] (g1h3) -- (g2h3) -- (g3h3) -- (g4h3);
    \draw[black,thick] (g1h1) to [out=-105,in=105] (g4h1);
    \draw[black,thick] (g1h2) to [out=-105,in=105] (g4h2);
    \draw[black,thick] (g1h3) to [out=-75,in=75] (g4h3);
    \draw[black,thick] (g1h1) -- (g2h2) -- (g3h1) -- (g4h2) -- (g1h1);
    \draw[black,thick] (g1h2) -- (g2h1) -- (g3h2) -- (g4h1) -- (g1h2);
    \draw[black,thick] (g1h2) -- (g2h3) -- (g3h2) -- (g4h3) -- (g1h2);
    \draw[black,thick] (g1h3) -- (g2h2) -- (g3h3) -- (g4h2) -- (g1h3);
    \draw[black,thick] (g1h1) -- (g2h3) -- (g3h1) -- (g4h3) -- (g1h1);
    \draw[black,thick] (g1h3) -- (g2h1) -- (g3h3) -- (g4h1) -- (g1h3);

    \node[label=below:{$C_4 \circ P_3$}] at (1,-3.5) {};
    
\end{tikzpicture}

\end{multicols}
\caption{From left to right, the Cartesian, Tensor, Strong, and Lexicographic products of $G$ and $H$ where $G = C_4$ and $H = P_3$.} \label{fig:products}
\end{figure}
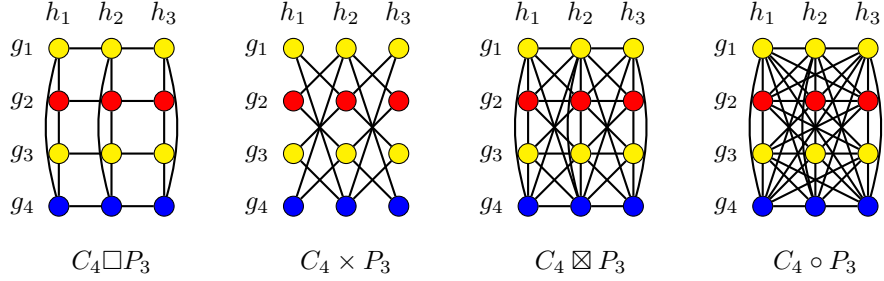

\end{example}

We now turn to a final graph construction, the join.  Recall that the \textit{join} of two graphs $G$ and $H$, denoted $G\nabla H$, has vertex set $V(G)\cup V(H)$ and edge set $E(G)\cup E(H) \cup \{gh \mid g \in V(G) \text{ and } h \in V(H)\}$.  As with the preceding products, the join of a square graph and an arbitrary graph is a square. Furthermore, we are able to give a necessary and sufficient condition for the join of any two non-square graphs to be a square.   Thus, we fully classify all graph joins.

\begin{theorem} \label{thm:join} Let $G$ and $H$ be graphs, each with at least one vertex. Then $G\nabla H$ is a square graph if and only if at least one of $G,H$ is a square graph or is not connected and has two isomorphic components $C_1, C_2$. 
\end{theorem}

\begin{proof}
Suppose that $G\nabla H$ is a square graph, and let $\phi, F_{\phi}, A_0, A_1$ be defined as in Observation~\ref{obs:alt-defn}. Since there are no edges between vertices in $A_0$ and $A_1$, and by definition of join, every vertex in $G$ is adjacent to every vertex in $H$, we have $A_0$ and $A_1$ are both subsets of $V(G)$ or both subsets of $V(H)$. Without loss of generality, suppose $A_0,A_1\subseteq V(G)$. By Observation~\ref{obs:alt-defn} part 4, we have $V(H)\subseteq F_{\phi}$. Thus, the restriction of $\phi$ to $V(G)$ is an involution that satisfies parts 2, 3, 4. If there is a vertex $g\in V(G)\cap F_{\phi}$, then $G$ is a square graph. If $V(G)=A_0\cup A_1$, then since there are no edges between $A_0$ and $A_1$, $G$ is not connected, and a component $C_1$ of the subgraph of $G$ induced by $A_0$ is sent by $\phi$ to an isomorphic component $C_2=\phi(C_1)$ in $A_1$. 

Before the proof of the converse, we note that because every vertex in $G$ is adjacent to every vertex in $H$ in $G\nabla H$, any automorphism of $G$ can be extended to an automorphism of $G\nabla H$ by fixing every vertex of $H$. 
Now suppose that $G$ is a square. Define $\phi$ to be an involution of $G$ that satisfies the conditions of Observation~\ref{obs:alt-defn}. Define $\psi $ to be the involution of $G\nabla H$ that fixes each vertex in $H$ and acts on $G$ by $\phi$. We have $F_{\psi}$ equals $F_{\phi}\cup V(H)$, and $A_0$ and $A_1$ are subsets of $V(G)$, defined from $\phi$. Thus, $\psi$ satisfies the conditions of Observation~\ref{obs:alt-defn} and shows that $G\nabla H$ is a square graph. 

Secondly, suppose that $G$ has two isomorphic components $C_1, C_2$. Define $\phi$ to be the involution that interchanges $C_1$ and $C_2$ and fixes the other vertices of $G$. Define $\psi $ to be the involution of $G\nabla H$ that fixes each vertex in $H$ and acts on $G$ by $\phi$. Then $F_{\psi}$ equals $(V(G)-C_1-C_2)\cup V(H)$, which is not empty since $H$ has at least one vertex, and $A_0=C_1, A_1=C_2$. Thus the conditions of Observation~\ref{obs:alt-defn} are satisfied, and $G\nabla H$ is a square.
\end{proof}

We observe that complete multipartite graphs can be constructed by a sequence of join operations, with each new factor an independent set. Thus, Proposition~\ref{pr:multipartite} follows from Theorem~\ref{thm:join}. We provide two more corollaries. 
\begin{corollary}
    \label{cor:dominating} 
Let $G$ be a connected graph. Then $G$ is a square if and only if $G \nabla K_1$ is a square.
\end{corollary}

\begin{proof}
By Theorem~\ref{thm:join}, we have $G\nabla K_1$ is a square if and only if $G$ is a square or $G$ is not connected and has two isomorphic connected components. By hypothesis $G$ is connected, so $G$ must be a square graph. 
\end{proof}

Let $m,n$ be positive integers. Recall that the fan graph $F_{m,n}$ is the join $I_m\nabla P_n$, where $I_m$ is an independent set of $m$ vertices.

\begin{corollary} \label{cor:fans}
   The fan graph $F_{m,n}$ is a square if and only if $m\geq 2$, or $m=1$ and $n\geq 3$ is odd.
\end{corollary}

\begin{proof} 

If $m\geq 2$, then $I_m$ has at least two isolated vertices which form isomorphic components, so by Theorem~\ref{thm:join}, $F_{m,n}$ is a square. 
If $m=1$ and $n\geq 3$ is odd, then $P_n$ is a square by Proposition~\ref{prop:paths}, so by Theorem~\ref{thm:join}, $F_{m,n}$ is a square. 

If $m=1$, and $n=1,2$ the graphs $F_{1,1}=K_2$ and $F_{1,2}=K_3$ are not squares, since they are complete graphs. 
If $m=1$ and $n\geq 3$ is even, then both $I_1=K_1$ and $P_n$ are connected. However, $K_1$ is not a square, and by Proposition~\ref{prop:paths}, $P_n$ is not a square, so by Theorem~\ref{thm:join}, $F_{m,n}$ is not a square. 
\end{proof}

\section{Conclusion} \label{sec:conclusion}

This paper takes the first steps at exploring squareness (under the gluing algebra) as a graph property of intrinsic interest.  It does so by first reframing the definition of a square graph in terms of butterfly involutions and symmetry, and then giving proofs for specific classical graph families that are and are not square. Naturally, many questions remain.  

\begin{ques} What other graph properties determine squareness?  For example, two vertices with the same open neighborhood is a sufficient but not necessary condition of squareness (Corollary \ref{cor:twins}).  It would be of interest to find other substructures that are sufficient to cause squareness.  Our results about graph products of square graphs make some progress in this direction. 
\end{ques}

\begin{ques} 
As we saw in Theorem~\ref{thm:circ-1-3}, circulants $C_8^{1,3}$ and $C_{10}^{1,3}$ are square graphs, but if $n\geq 11$, $C_{n}^{1,3}$ is not a square graph. 
For $d\geq 4$, in Theorem~\ref{thm:circ-1-d}, we showed that circulants $C_n^{1,d}$ are not square graphs for $n\geq d^2+1$. An open question is to complete the characterization of these circulants for $n\leq d^2$. 
\end{ques}

\begin{ques}
For four variations of graph products, we proved that the product of a square and arbitrary graph is a square.  However, the squareness of the product of two non-squares is still open.  When is it possible for two non-square graphs to produce a square?
\end{ques}

\section*{Acknowledgments}
This work was started at the workshop ``Graph Theory: structural properties, labelings, and connections to applications'' hosted by the American Institute of Mathematics (AIM), Pasedena CA, July 22--July 26, 2024. The authors thank AIM and the organizers of the workshop for facilitating their collaboration, as well as fellow workshop attendees for early discussions (especially Miranda Bowie).  The authors also thank Annie Raymond for many productive conversations and useful references. 

\bibliography{bib}

\end{document}